\documentclass{amsart}
\usepackage{hyperref}

\usepackage{amsrefs,amsmath,amsfonts,amssymb,graphicx,latexsym,color}

\newcommand{\relMOM}{g}

\renewcommand{\dim}{3}
\newcommand{\threed}{{\mathbb R}^\dim}

\newcommand{\R}{{\mathbb R}}
\newcommand{\sph}{{\mathbb S}^{2}}

\newcommand{\eqdef}{\overset{\mbox{\tiny{def}}}{=}}

\newcommand{\ang}[1]{ \left< {#1} \right> }

\newcommand{\nullSpace}{N(L)}

\newcommand{\pa}{\partial}
\newcommand{\be}{\beta}
\newcommand{\ag}{\gamma}

\newcommand{\pZ}{p^{0}}
\newcommand{\qZ}{q^{0}}
\newcommand{\pZprime}{p^{\prime 0}}
\newcommand{\qZprime}{q^{\prime 0}}

\newcommand{\macroCOE}{\lambda}

\newcommand{\rP}{p^{\prime}}
\newcommand{\rQ}{q^{\prime}}
\newcommand{\PrP}{p^{\prime\prime}}
\newcommand{\QrQ}{q^{\prime\prime}}

\newcommand{\CD}{\mathcal{D}}
\newcommand{\CE}{\mathcal{E}}

\newtheorem{theorem}{Theorem}

\newtheorem{proposition}{Proposition}
\newtheorem{lemma}[proposition]{Lemma}
\theoremstyle{definition}

\numberwithin{equation}{section}

\begin{document}
\title[Relativistic Vlasov-Maxwell-Boltzmann System]{Momentum Regularity and
Stability of the Relativistic Vlasov-Maxwell-Boltzmann System}
\author[Y. Guo]{Yan Guo}
\address{(Y.G.) Division of Applied Mathematics, Brown University, Providence,
RI 02912}
\email{guoy at dam.brown.edu}
\thanks{Y.G. was partially supported by the NSF grant DMS-0905255, a NSF FRG
grant, and the Chinese NSF grant \# 10828103.}
\author[R. M. Strain]{Robert M. Strain}
\address{(R.M.S.) University of Pennsylvania, Department of Mathematics, David
Rittenhouse Lab, 209 South 33rd Street, Philadelphia, PA 19104}
\email{strain at math.upenn.edu}
\urladdr{http://www.math.upenn.edu/~strain/}
\thanks{R.M.S. was partially supported by the NSF grant DMS-0901463, and an Alfred P. Sloan Foundation Research Fellowship.}

%\keywords{Relativity, Boltzmann, relativistic Maxwellian, stability,
%Newtonian Limit, collisional Kinetic Theory, Kinetic Theory}
%\subjclass[2000]{Primary: 76P05; Secondary: 83A05}

\begin{abstract}
In the study of solutions to the relativistic Boltzmann equation, their
regularity with respect to the momentum variables has been an outstanding
question, even local in time, due to the initially unexpected growth in the
post-collisional momentum variables which was discovered in 1991 by Glassey
\& Strauss \cite{MR1105532}. We establish momentum regularity within energy
spaces via a new splitting technique and interplay between the
Glassey-Strauss frame and the center of mass frame of the relativistic collision
operator. In a periodic box, these new momentum regularity estimates lead to
a proof of global existence of classical solutions to the two-species
relativistic Vlasov-Maxwell-Boltzmann system for charged particles near
Maxwellian with hard ball interaction.
\end{abstract}

% set the depth for the table of contents (0-2)
\setcounter{tocdepth}{1}
\maketitle
\tableofcontents

\thispagestyle{empty}

\section{Introduction and formulation}

In 2003 it was shown for the first time that the full (Newtonian) Vlasov-Maxwell-Boltzmann system 
\cite{MR2000470} has global in time unique classical solutions on the torus for initial conditions which are sufficiently close to the steady state.  Then in 2006 the rapid convergence \cite{MR2209761} on the torus for these solutions, as predicted by the H-theorem, was established.  However it should be pointed out that this model is physically limited because it is not fully Lorentz invariant in the sense that the symmetries of the Maxwell system are inconsistent with those of the Newtonian Boltzmann equation and serious difficulties are encountered in extending the method from \cite{MR2000470} to the relativistic fully Lorentz invariant regime.  In this work, as explained in the following developments, we overcome these difficulties and establish the momentum regularity as well as the existence of global in time classical solutions for the fully Lorentz invariant relativistic Vlasov-Maxwell-Boltzmann system with hard-ball interaction.

We study the following two-species relativistic Vlasov-Maxwell-Boltzmann system which 
describes the time evolution of charged particles: 
\begin{equation}
\begin{split}
\partial _{t}F_{+}+c\frac{p}{p_{+}^{0}}\cdot \nabla _{x}F_{+}
+
e_{+}\left( E+\frac{p}{p_{+}^{0}}\times B\right) \cdot \nabla _{p}F_{+}& 
=
\mathcal{Q}^{+}(F)+\mathcal{Q}^{\pm }(F), 
\\
\partial _{t}F_{-}+c\frac{p}{p_{-}^{0}}\cdot \nabla _{x}F_{-}-e_{-}
\left( E+\frac{p}{p_{-}^{0}}\times B\right) \cdot \nabla _{p}F_{-}& 
=
\mathcal{Q}^{-}(F)+\mathcal{Q}^{\mp }(F).
\end{split}
\label{rVMB}
\end{equation}%
The collision operators, defined in \eqref{rCOL} below, are given by $\mathcal{Q}^{+}(F)\eqdef \mathcal{Q}%
(F_{+},F_{+}),$ $\mathcal{Q}^{-}(F)\eqdef \mathcal{Q}%
(F_{-},F_{-}),$ $\mathcal{Q}^{\pm }(F)\eqdef \mathcal{%
Q}(F_{+},F_{-}),$ and $\mathcal{Q}^{\mp }(F)\eqdef %
\mathcal{Q}(F_{-},F_{+})$.  

 These equations are coupled with the Maxwell system: 
\begin{eqnarray*}
\partial _{t}E-c\nabla _{x}\times B &=&-4\pi \int_{{\mathbb{R}}^{\dim
}}\left\{ e_{+}\frac{p}{p_{+}^{0}}F_{+}-e_{-}\frac{p}{p_{-}^{0}}%
F_{-}\right\} dp, \\
\partial _{t}B+c\nabla _{x}\times E &=&0,
\end{eqnarray*}%
with constraints 
\begin{equation*}
\nabla _{x}\cdot E=4\pi \int_{{\mathbb{R}}^{\dim}}\left\{ e_{+}F_{+}-e_{-}F_{-}\right\} dp,
\quad 
\nabla _{x}\cdot B=0.
\end{equation*}%
The initial conditions are $F_{\pm }(0,x,p)=F_{0,\pm }(x,p)$, $E(0,x)=E_{0}(x)$,
and $B(0,x)=B_{0}(x)$.   Here $F_{\pm }(t,x,p)\geq 0$ are the spatially
periodic number density functions for ions $(+)$ and electrons $(-)$, at
time $t\geq 0$, position $x=(x_{1},x_2 ,x_{\dim })\in {\mathbb{T}}^{\dim }%
\eqdef [-\pi ,\pi ]^{\dim }$ and momentum $%
p=(p_{1},p_2 ,p_{\dim })\in {\mathbb{R}}^{\dim }$. The constants $\pm
e_{\pm }$ and $m_{\pm }$ are the magnitude of the particles charges and rest
masses respectively. The energy of a particle is given by 
$p_{\pm }^{0}=\sqrt{(m_{\pm }c)^{2}+|p|^{2}}$ and $c$ is the speed of light.  Note that here and below $\pm$ indicates two possible sign configurations.

For number density functions $F_+(p)$ and $F_-(p)$ a collision operator
should satisfy 
\begin{equation*}
\int_{{\mathbb{R}}^\dim} \left\{ 
\begin{pmatrix}
1 \\ 
p \\ 
p_+^0%
\end{pmatrix}%
\mathcal{Q}(F_+, F_-)(p) + 
\begin{pmatrix}
1 \\ 
p \\ 
p_-^0%
\end{pmatrix}
\mathcal{Q}(F_-,F_+)(p) \right\} dp = 0.
\end{equation*}
The same property holds for the other sign configurations. By integrating the
relativistic Vlasov-Maxwell-Boltzmann system and plugging in this identity,
we obtain the conservation of mass, total momentum and total energy for
solutions as 
\begin{equation}
\begin{split}
& \frac d{dt}\int_{{\mathbb{T}}^\dim\times {\mathbb{R}}^\dim}m_{+}F_{+}(t)
= \frac d{dt}\int_{{\mathbb{T}}^\dim\times {\mathbb{R}}^\dim}m_{-}F_{-}(t)=0,
\\
& \frac d{dt}\left\{ \int_{{\mathbb{T}}^\dim\times {\mathbb{R}}^\dim}
p(m_{+}F_{+}(t)+m_{-}F_{-}(t))+\frac 1{4\pi }\int_{{\mathbb{T}}%
^\dim}E(t)\times B(t)\right\} =0, \\
& \frac d{dt}\left\{ \frac 12\int_{{\mathbb{T}}^\dim\times {\mathbb{R}}%
^\dim}(m_{+}p_+^0F_{+}(t) +m_{-}p_-^0F_{-}(t))+{\frac 1{8\pi }}\int_{{%
\mathbb{T}}^\dim}|E(t)|^2+|B(t)|^2\right\} =0.
\end{split}
\label{conserv}
\end{equation}
The entropy of the relativistic Vlasov-Maxwell-Boltzmann system is defined as 
\begin{equation}
\mathcal{H}(t) \eqdef  -\int_{{\mathbb{T}}^\dim\times 
{\mathbb{R}}^\dim} dx dp ~ \left\{F_{+}(t, x, p)\log F_{+}(t, x, p) +
F_{-}(t, x, p)\log F_{-}(t, x, p) \right\}.  
\notag
\end{equation}
Then the celebrated Boltzmann H-theorem for the relativistic Vlasov-Maxwell-Boltzmann system
corresponds to the following formal statement 
\begin{equation}
\frac d{dt}\mathcal{H}(t) \ge 0,  \notag
\end{equation}
which says that the entropy of solutions is non-decreasing as time passes. 

The global relativistic Maxwellian (a.k.a. the J\"{u}ttner solution) is
given by 
\begin{equation*}
J_{\pm }(p)\eqdef  \frac{\exp \left( -c\pZ _{\pm
}/(k_{B}T_{\pm })\right) }{4\pi e_{\pm }m_{\pm }^{2}ck_{B}T_{\pm
}K_{2}(m_{\pm }c^{2}/(k_{B}T_{\pm }))},
\end{equation*}%
where $k_B > 0$ denotes \emph{Boltzmann's constant}, $K_{2}(\cdot )$ is the
Bessel function  $K_{2}(z)\eqdef \frac{z^{2}}{2}\int_{1}^{\infty
}e^{-zt}(t^{2}-1)^{3/2}dt$, and $T_{\pm }$ is the temperature. From the Maxwell
system %$(\ref{maxwell}) 
and the periodic boundary condition of $E(t,x)$, we see that $\frac{d}{dt}%
\int_{{\mathbb{T}}^\dim}B(t,x)dx\eqdef 0.$ We thus
have a constant $\bar{B}$ such that 
\begin{equation}
\frac{1}{|{\mathbb{T}}^\dim|}\int_{{\mathbb{T}}^\dim}B(t,x)dx=\bar{B}.
\label{bar}
\end{equation}%
Let $[\cdot ,\cdot ]$ denote a column vector. We then have the following
steady state solution to the relativistic Vlasov-Maxwell-Boltzmann system 
\begin{equation*}
\lbrack F_{\pm }(t,x,p),E(t,x),B(t,x)]=[J_{\pm },0,\bar{B}],
\end{equation*}%
which maximizes the entropy.

We furthermore define the relativistic Boltzmann collision operator \cite{MR635279} as  
\begin{equation}
Q(F_\pm, G_\mp) \eqdef
 \int_{\threed} \frac{dq}{\qZ_{\mp}}
\int_{\threed}\frac{dq^\prime}{\qZprime_{\mp}}
\int_{\threed}\frac{dp^\prime}{\pZprime_{\pm}}
~ W_{\pm |\mp} ~ [F_\pm(\rP)G_\mp(\rQ)-F_\pm(p)G_\mp(q)].  
\label{rCOL}
\end{equation}
It is written similarly for other sign configurations.  Here the ``transition rate'' $W_{\pm |\mp}
=
W_{\pm |\mp}(p, q | p^\prime, q^\prime)$ is defined as
\begin{equation}
W_{\pm |\mp}  %(p, q | p^\prime, q^\prime)
\eqdef
 s ~ \sigma_{\pm |\mp}(\relMOM, \theta) ~
\delta(\pZ_{\pm}+\qZ_{\mp}-\pZprime_{\pm}-\qZprime_{\mp})
\delta^{(\dim)}(p+q-p^\prime-q^\prime).    \label{transition}
\end{equation}
The quantities here are $s=s(p_{\pm}, q_{\mp})$, $\relMOM=\relMOM(p_{\pm}, q_{\mp})$, $\theta=\theta(p_{\pm}, q_{\mp})$ and $\sigma_{\pm |\mp}(\relMOM, \theta)$ are defined exactly as in the following  sub-section.

The physical intuition provided by the Boltzmann H-theorem is to say that solutions should converge to their steady state, which is chosen by the initial conditions and the conservation laws \eqref{conserv}, as time goes to infinity.  Our goal in this work is to prove this global existence and rapid convergence in the context of perturbations.

We define the standard perturbation $f_{\pm }(t,x,p)$ to $J_\pm$ as 
\begin{equation*}
F_{\pm }\eqdef  J_{\pm } +\sqrt{J }_{\pm }f_{\pm }.
\end{equation*}
We will plug this ansatz into \eqref{rVMB} to derive a perturbed system for $f_{\pm }(t,x,p)$, $E(t,x)$
and $B(t,x)$. The two relativistic Vlasov-Maxwell-Boltzmann equations for
the perturbation $f=[f_+, f_-]$ take the form 
\begin{gather}
\left\{ \partial_t+c\frac{p}{\pZ_\pm}\cdot \nabla_x
\pm e_\pm\left(E+\frac{p}{\pZ_\pm}\times B\right)\cdot \nabla_p\right\} f_\pm \mp\frac{e_\pm c}{k_BT}%
\left\{E\cdot \frac{p}{\pZ_\pm}\right\} \sqrt{J}_\pm+L_\pm f  \notag \\
= \pm \frac{e_\pm c}{2 k_BT}\left\{E\cdot \frac{p}{\pZ_\pm}\right\}f_\pm
+\Gamma_\pm(f,f),  \label{rvmlC}
\end{gather}
with $f(0, x, p)=f_0(x,p)=[f_{0,+}(x,p),f_{0,-}(x,p)] $. The linear operator 
$L_\pm f$, defined in (\ref{L}), and the nonlinear operator $\Gamma_\pm(f,f)$,
defined in \eqref{gamma0}, are derived from an expansion of the Boltzmann
collision operator \eqref{rCOL}. 

In particular, 
using \eqref{transition}, we
observe that the collision operator \eqref{rCOL} satisfies 
\begin{equation*}
\mathcal{Q}(J_+, J_+) =\mathcal{Q}(J_+, J_-) =\mathcal{Q}(J_-, J_+) =%
\mathcal{Q}(J_-, J_-) =0.
\end{equation*}
Then, with $f=[f_+, f_-]$ and $h=[h_+, h_-]$, 
we can write the nonlinear operators as
\begin{equation}
\Gamma_\pm (f,h)  \eqdef
J_\pm^{-1/2}\mathcal{Q}(\sqrt{J}_\pm f_\pm,\sqrt{J}_\pm h_\pm) 
+ 
J_\pm^{-1/2}
\mathcal{Q}(\sqrt{J}_\pm f_\pm,\sqrt{J}_\mp h_\mp).
\label{gamma0}
\end{equation}
Furthermore the linearized collision operators take the form
\begin{equation}
L_\pm (h)  \eqdef
-\Gamma_\pm (h,\sqrt{J}) 
- 
\Gamma_\pm (\sqrt{J},h) .
\label{L}
\end{equation}
We estimate these operators in Sections \ref{sec:momD} and \ref{sec:linearE}.

In this linearized system, 
the coupled Maxwell system is given by
\begin{equation}
\begin{split}
\partial _tE-c\nabla_x \times B &= -4\pi \mathcal{J} \overset{%
\mbox{\tiny{def}}}{=} -4\pi \int_{{\mathbb{R}}^\dim}\left\{e_+\frac
p{\pZ _+}\sqrt{J}_+f_{+}-e_-\frac p{\pZ _-}\sqrt{J}_-f_{-}\right\}dp, \\
\partial _tB + c\nabla_x \times E &=0,
\end{split}
\label{maxwellC}
\end{equation}
with constraints 
\begin{gather}
\nabla_x \cdot E=4\pi\rho \eqdef  4\pi\int_{{\mathbb{R}}^\dim}\left\{e_+\sqrt{J}_+f_{+}-e_-\sqrt{J}_-f_{-}\right\}dp, \quad
\nabla_x \cdot B=0.  \label{constraintC}
\end{gather}
The charge density and current density
due to particles are denoted $\rho$ and $\mathcal{J}$ respectively. In
computing $\rho$, we have used the normalization $\int_{{\mathbb{R%
}}^\dim} J_\pm(p) dp = \frac{1}{e_\pm}. $

{\bf In all of the following developments, none of the physical constants will affect the results of our analysis.  Then without loss of generality but for the sake of simplicity, we  normalize all constants to one.  We set   }
$$
J_\pm (p) = J(p) = \frac{e^{-\pZ}}{4\pi}.
$$
{\bf We drop the inessential $\pm$ notation all over, in particular 
we use the kernel notation $\sigma(\relMOM, \theta)=\sigma_{\pm |\mp}(\relMOM, \theta)$ and we set $\pZ_\pm = \pZ \eqdef \sqrt{1+|p|^2}$, etc.}

Furthermore we assume that initially $[F_{0},E_{0,}B_{0}]$ has the same mass, total momentum and total energy %
\eqref{conserv} as the steady state $[J,0,\bar{B}]$, then we can
rewrite the conservation laws in terms of the perturbation $[f,E,B]$ as
follows: 
\begin{gather}
\int_{\mathbb{T}^{3}\times \mathbb{R}^{3}}dxdp~f_{+}(t)\sqrt{J} =\int_{%
\mathbb{T}^{3}\times \mathbb{R}^{3}}dxdp~f_{-}(t)\sqrt{J}=0,  \label{ma} 
\\
\int_{\mathbb{T}^{3}\times \mathbb{R}^{3}}dxdp~p\left\{ f_{+}(t) + f_{-}(t)\right\}\sqrt{J(p)}  =-\frac{1}{4\pi }\int_{\mathbb{T}%
^{3}}dx~E(t)\times B(t),  \label{mo} 
\\
\int_{\mathbb{T}^{3}\times \mathbb{R}^{3}}dxdp~\pZ \left\{ f_{+}(t)
+ f_{-}(t) \right\}\sqrt{J} =-\frac{1}{8\pi }\int_{\mathbb{T}%
^{3}}dx~|E(t)|^{2}+|B(t)-\bar{B}|^{2}.  \label{en}
\end{gather}%
We have used (\ref{bar}) for the normalized energy conservation \eqref{en}. 

In the next sub-section we will discuss reductions of the collision operator \eqref{rCOL}.

\subsection{Collision operator in the Glassey-Strauss frame}

In a pioneering work of Glassey and Strauss \cite{MR1211782}, the collision
operator $\mathcal{Q}$ was represented as follows: 
\begin{equation}
\mathcal{Q}(f,h)
=
\int_{{\mathbb{R}}^{\dim
}\times {\mathbb{S}}^{2}}\frac{s\sigma (g,\theta )}{\pZ \qZ}B(p,q,\omega
)[f(\rP)h(\rQ)-f(p)h(q)]d\omega dq,  
\label{collisionGS}
\end{equation}%
where the kernel is 
\begin{equation}
B(p,q,\omega )\eqdef \frac{(\pZ +\qZ)^{2}\pZ \qZ\left\vert \omega \cdot \left( \frac{p}{\pZ }-%
\frac{q}{\qZ}\right) \right\vert }{\left[ (\pZ +\qZ)^{2}-(\omega \cdot
\lbrack p+q])^{2}\right] ^{2}}.  \label{kernelGS}
\end{equation}%
Above the quantity $s=s(p,q)$, which is the square of the energy in the \textquotedblleft
center of momentum\textquotedblright\ system, $p+q=0$, is defined as 
\begin{equation*}
s\eqdef 2(\pZ \qZ-p\cdot q+1).
\end{equation*}%
The relative momentum, $\relMOM=\relMOM (p,q)$, is denoted 
\begin{equation}
\relMOM\eqdef \sqrt{2(\pZ \qZ-p\cdot q-1)}.
\label{gDEFINITION}
\end{equation}
Notice that $s=g^{2}+4$. We point out that this notation, which
is used in \cite{MR635279}, may differ from other authors notation by a
constant factor.

The condition for elastic collisions is then given by 
\begin{equation}
\begin{array}{ccc}
\pZ+ \qZ & = & \pZprime+\qZprime, 
\\ 
p+q & = & p^\prime+q^\prime.
\end{array}
\label{collisionalCONSERVATION}
\end{equation}
In this expression, the post collisional momentum are given as follows 
\begin{equation}
\begin{array}{ccc}
\rP & = & p+a(p,q,\omega )\omega, \\ 
\rQ & = & q-a(p,q,\omega )\omega,%
\end{array}
\label{postCOLLvelGS}
\end{equation}
where 
\begin{equation*}
a(p,q,\omega )\eqdef  \frac{2(p^0+q^0)p^0q^0\left\{
\omega \cdot \left( \frac{q}{q^0}-\frac{p}{p^0}\right) \right\} }{%
(p^0+q^0)^{2}-\left\{ \omega \cdot \left( p+q\right) \right\} ^{2}}.
\end{equation*}
The Jacobian for the transformation $(p, q)\to (p^\prime, q^\prime)$ in
these variables \cite{MR1105532} is 
\begin{equation}
\frac{\partial (\rP,\rQ)}{\partial (p,q)}=- \frac{\pZprime\qZprime}{p^0q^0}.  \label{PARTICLEjacobian}
\end{equation}
Now we turn to a discussion of the collision kernel $\sigma (g, \theta)$ in \eqref{collisionGS}.   The kernel $\sigma(g, \theta)$ measure's the interactions between particles. See 
\cite{DEnotMSI,MR933458} for a physical discussion of general assumptions.  We use the following hypothesis.
\newline

\noindent \textbf{Hypothesis on the collision kernel:} We consider the
``hard ball'' condition 
\begin{equation}
\sigma (g,\theta)=\text{constant}.  \notag
\end{equation}%
This condition is implicitly used throughout the rest of the article. In
fact to simplify the notation, without loss of generality, in the following
we use the normalized condition $\sigma (g,\theta)=1$. The Newtonian limit,
as $c\rightarrow \infty $, in this situation is the Newtonian hard-sphere
Boltzmann collision operator \cite{strainNEWT}. \newline

In the next section we will discuss our main results.

\section{Statement of the main results}

Let the multi-indices $\ag$ and $\beta$ be $\ag =[\ag^{0}, \ag ^{1},\ag
^{2},\ag ^{3}]$ and $\beta =[\beta ^{1},\beta ^{2},\beta ^{3}].$ We will use the
following notation for a high order derivative 
\begin{equation*}
\partial _{\beta }^{\ag }\eqdef 
\partial_{t}^{\ag ^{0}}
\partial_{x_{1}}^{\ag ^{1}}\partial _{x_{2}}^{\ag ^{2}}\partial
_{x_{3}}^{\ag ^{3}}\partial _{p_{1}}^{\beta ^{1}}\partial _{p_{2}}^{\beta
^{2}}\partial _{p_{3}}^{\beta ^{3}}.
\end{equation*}%
We sometimes also use the notation $\beta_0$, $\beta_1$, $\beta_2$ and $\alpha$ to denote multi-indices with three components such as $\beta$.  Then if each component of $\beta $ is not greater than that of $\beta_1$, 
we denote this by $\beta \le \beta_1$, also $\beta <\beta_1$ means $\beta
\le \beta_1$ and $|\beta |<|\beta_1|$
where  $|\beta | = \beta^1 + \beta^2 + \beta^3$ as usual.
We also denote a collection of weight functions by 
\begin{equation*}
w_{\ell }
\eqdef 
\left\langle {p}\right\rangle^{\ell},
\quad
\ang{p} \eqdef
\sqrt{1+|p|^2},
\quad \ell \in {\mathbb{R}}.
\end{equation*}
Given a solution $[f_{\pm }(t,x,p),E(t,x),B(t,x)]$ to the relativistic
Vlasov-Maxwell-Boltzmann system \eqref{rvmlC}, \eqref{maxwellC}, and %
\eqref{constraintC}, we define the full instant energy functional to be a continuous function, $\CE_{N,\ell}(t)$, which satisfies:
\begin{equation}
\CE_{N,\ell}(t)\approx 
\sum_{|\ag |+|\beta |\leq N}\Vert w_{\ell}\partial _{\beta }^{\ag }f(t)\Vert ^{2}
+
\sum_{|\ag |\leq N}\Vert \partial ^{\ag }[E(t),B(t)]\Vert ^{2}.
\label{def.eNmN}
\end{equation}
At time $t=0$ the time derivatives in $\CE_{N,\ell}(0)$ are defined customarily through equations \eqref{rvmlC} and \eqref{maxwellC}.  In \eqref{def.eNmN} and the rest of this paper, the norm $\| \cdot \|$ denotes either the $L^2(\mathbb{T}^3_x\times\R^3_p)$ norm or just the $L^2(\mathbb{T}^3_x)$ without ambiguity depending upon the variables in the functions being measured.
Throughout the rest of this paper we furthermore assume $N\geq 4$ and $\ell \geq 0$.

We are now ready to state our main results:

\begin{theorem}
\label{mainTHM} Suppose that $[f_{0,\pm},E_{0},B_{0}]$ satisfies the constraint \eqref{constraintC} and the the conservation
laws \eqref{ma}, \eqref{mo}, \eqref{en} initially. Fix $\ell \geq 0$ and $N\geq 4$. Consider 
$
F_{0,\pm }(x,p)=J_{\pm }+\sqrt{J}_{\pm }f_{0,\pm }(x,p).
$
There is a constant $M>0$ such that if 
\begin{equation*}
\CE_{N,\ell}(0)\leq M,
\end{equation*}%
then there exists a unique global solution $[f_{\pm }(t,x,p),E(t,x),B(t,x)]$
to the perturbed relativistic Vlasov-Maxwell-Boltzmann system \eqref{rvmlC},
\eqref{maxwellC} with \eqref{constraintC} satisfying 
\begin{equation*}
\CE_{N,\ell}(t)\lesssim \CE_{N,\ell}(0).
\end{equation*}%
Moreover $F_{\pm }(t,x,p)=J_{\pm }+\sqrt{J}_{\pm }f_{\pm }(t,x,p)$
solves the relativistic Vlasov-Maxwell-Boltzmann system \eqref{rVMB}, 
and $F_{\pm }(t,x,p)\geq 0$ if this is so initially.

If $\CE_{N+k,\ell }(0)$ is sufficiently small for some $k>0$, then we
have rapid decay as 
\begin{equation*}
\sum_{|\ag |+|\beta |\leq N}\Vert w_{\ell}\partial _{\beta }^{\ag }f(t)\Vert ^{2}
+
\sum_{|\ag |\leq N}\Vert \partial ^{\ag }[E(t),B(t)-\bar{B}]\Vert ^{2}
\lesssim 
\CE_{N+k,\ell}(0)
\left( 1+t\right)^{-k},
\end{equation*}
where the constant $\bar{B}$ is defined in \eqref{bar}.
\end{theorem}

There have been many investigations about various kinetic models for
describing charged particles. Standard
references include \cite{MR1898707,MR635279,MR1379589}.
We refer to several results such as
\cite{guoWS,MR2217287,MR1402446,MR2098116,MR1151987,MR933458,MR2366140}.  In \cite{MR1908664}, global classical
solutions were constructed for the Vlasov-Poisson-Boltzmann system (no
magnetic fields) via introduction of a nonlinear energy method for which the
linear collision operator $L$ is positive for solutions near Maxwellians. In 
\cite{MR2000470}, an improvement of such a method led to construction of
global solutions to the Vlasov-Maxwell-Boltzmann system in the presence of
magnetic field. In \cite{MR2259206}, such a construction was carried out in the whole space case using some new dissipation estimates. Even though the Vlasov-Maxwell-Boltzmann system can be
viewed as a `master system' for kinetic models, from general physical
principle, the classical (non-relativistic) Boltzmann is not compatible
with the (Lorentz invariant) Maxwell system, which obeys special relativity. It is
therefore important to study the relativistic effects for the relativistic
Vlasov-Maxwell-Boltzmann system \eqref{rVMB} and to generalize the result in 
\cite{MR2000470} to a relativistic setting. However, such a project was
easily stopped due to a severe difficulty of lack of regularity in the momentum $p$ variables for the relativistic Boltzmann equation.

In the Glassey-Strauss frame \eqref{collisionGS} and \eqref{postCOLLvelGS},
the following pointwise estimates were discovered by Glassey \& Strauss \cite{MR1105532} in 1991:
\begin{equation}
\left\vert \nabla _{p}q_{i}^{\prime }\right\vert +\left\vert \nabla
_{p}p_{i}^{\prime }\right\vert \leq C\ang{q}^{5}\left( 1+|p\cdot \omega |^{1/2}%
\mathbf{1}_{\{|p\cdot \omega |>|p\times \omega |^{3/2}\}}\right) .
\label{GSgrowth}
\end{equation}%
This is a sharp estimate at least in regards to the $p$ growth. Furthermore 
\begin{equation*}
\left\vert \nabla _{q}p_{i}^{\prime }\right\vert +\left\vert \nabla
_{q}q_{i}^{\prime }\right\vert \leq C\ang{q}^{5}\pZ .
\end{equation*}%
Although in this second estimate, no attempt was made to further refine it.
Notice that higher derivatives do not exhibit momentum growth in $p$: 
\begin{equation*}
\left\vert \nabla _{p}^{k}p_{i}^{\prime }\right\vert +\left\vert \nabla
_{p}^{k}q_{i}^{\prime }\right\vert +\left\vert \nabla _{q}^{k}p_{i}^{\prime
}\right\vert +\left\vert \nabla _{q}^{k}q_{i}^{\prime }\right\vert \leq
C\ang{q}^{5},\quad \forall k>1.
\end{equation*}%
The $q$ growth above does not cause any difficulty because we have strong
exponential decay in that variable in the linearized collision operator $\Gamma $. But the $p$ momentum growth
in \eqref{GSgrowth}, introduces high order growth of $p$ in $\{\left\vert
\nabla _{p}q_{i}^{\prime }\right\vert +\left\vert \nabla _{p}p_{i}^{\prime
}\right\vert \}^{N}$, within the highest order derivatives of $\partial
_{\beta }f(\rP)$ with $|\beta |=N.$

Such a growth phenomenon is purely a relativistic effect, which is absent in
the classical setting. Consequently, no regularity for the momentum variables, even
local in time, for the solutions of the Boltzmann equation has been
obtained. This is because of the presence of highest polynomial growth rate in $p$ in
the highest order derivatives of the solutions, which makes it impossible to close the 
estimates in any reasonable energy norm.  Up to now, all existing work for the
relativistic Boltzmann equation only involves spatial and temporal
regularity.  Unfortunately, it is necessary to obtain momentum regularity to
study the relativistic Vlasov-Maxwell-Boltzmann system due to the Lorentz
force term: $\left( E+\hat{p}\times B \right)\cdot \nabla _{p}f.$

\subsection{Collision operator in the center of mass frame}

Our key idea to overcome such a severe difficulty in the Glassey-Strauss
frame is to resort to the following center of mass representation of
the Boltzmann collision operator. We remark that the study of such a center
of mass frame was initiated recently in the absence of momentum derivatives 
\cite{strainCOOR,strainPHD}.   For a function $G:{\mathbb{R}}^{\dim }\times 
{\mathbb{R}}^{\dim }\times {\mathbb{R}}^{\dim }\times {\mathbb{R}}^{\dim
}\rightarrow {\mathbb{R}}$:
\begin{equation}
\int_{{\mathbb{S}}^{2}}d\omega ~\frac{sB(p,q,\omega )}{\pZ \qZ}%
~G(p,q,\rP,\rQ)=\int_{{\mathbb{S}}^{2}}d\omega ~v_{\o %
}~G(p,q,\PrP,\QrQ),  \notag
\end{equation}%
where $B(p,q,\omega )$ is given by \eqref{kernelGS} and 
$(\rP,\rQ)$ on the l.h.s. are given by \eqref{postCOLLvelGS}. On the
r.h.s. we use that $v_{\o}=v_{\o}(p,q)$ is the M\o ller velocity given by 
\begin{equation}
v_{\o }\eqdef \frac{1}{2}\frac{g\sqrt{s}}{\pZ \qZ}.
\label{moller}
\end{equation}%
The post-collisional momentum, $(\PrP,\QrQ)$, on
the r.h.s. can be written: 
\begin{equation}
\begin{split}
\PrP& \eqdef \frac{p+q}{2}+\frac{g}{2}\left( \omega +(\rho -1)(p+q)%
\frac{(p+q)\cdot \omega }{|p+q|^{2}}\right) , \\
\QrQ& \eqdef \frac{p+q}{2}-\frac{g}{2}\left( \omega +(\rho -1)(p+q)%
\frac{(p+q)\cdot \omega }{|p+q|^{2}}\right),
\label{pcM}
\end{split}%
\end{equation}%
where $\rho =(\pZ +\qZ)/\sqrt{s}$. See \cite[Corollary 5]{strainCOOR}
for basic properties of such a center of mass frame.    
In particular the Jacobian \eqref{PARTICLEjacobian} effectively also works here as
\begin{multline}\label{changeCM}
\int_{\threed} dp \int_{\threed} dq
\int_{\sph}d\omega
 ~ v_{\o } ~ \sigma(\relMOM, \theta) ~ G(p, q, \PrP, \QrQ)
\\
=
\int_{\threed} dp \int_{\threed} dq
\int_{\sph}d\omega 
 ~ v_{\o } ~ \sigma(\relMOM, \theta) ~ 
G(\PrP, \QrQ, p,q).
\end{multline}
A more detailed explanation is given in \cite[Corollary 5 and (23)]{strainCOOR}.

Clearly, there is also a problem in the center of mass variables from \eqref{pcM}. In these
variables it is straightforward to compute that high momentum derivatives of 
$\PrP$ and $\QrQ$ create high singularities when 
$p=q$ and $p=-q$. These two distinct problems in each separate
representation formula for the relativistic Boltzmann collision operator
illustrate the main reason why it has remained an open problem to prove
energy estimates with momentum regularity.

To resolve these difficulties, and to prove the main nonlinear estimate in Theorem \ref{thm:nonLINest} below, we
will split the desired estimate into two different cases. These cases
correspond to the following two different integration regions: 
\begin{equation}
A\eqdef \{|p|\leq 1\}\cup \{|p|\geq 1,|p|^{\frac{1}{m}}\leq 2\qZ\},
\quad 
A_{c}\eqdef \{|p|\geq 1,|p|^{1/m}\geq 2\qZ\}.  
\label{ABsplitting}
\end{equation}
Here $m\gg 1$ is taken to be a given large positive integer. On the set $A,$
we can use the Glassey-Strauss frame \eqref{postCOLLvelGS}. Large growing
polynomial momentum weights in $p$, as described above and in \eqref{GSgrowth} can be controlled by the factor $J^{1/4}(q)$ since $|p| \lesssim (\qZ)^{m}$ on $A$. On the other hand, on the region $A_{c},$ we will use
the center-of-momentum variables \eqref{pcM}. Note $|p|\geq
2|q|$ so that $|p\pm q|\geq \frac{|p|}{2}\geq \frac{1}{2}$. Then the deficiency
in these variables, namely that derivatives of \eqref{pcM}
create singularities (even though there is no momentum growth at infinity),
is fortunately avoided on the region $A_{c}$; meaning that our estimates in
this region are safe as well. Without such a magical use of the center of
mass frame, it is still an outstanding open question if one can control the
growth in $p$ solely within the Glassey-Strauss frame.

We would like to say that we think it would be interesting to study this analogous problem in the whole space $\R^3_x$  including the convergence rates, as in \cite{MR2259206,DS-VPB,DS-VMB,szSOFTwhole}.

\subsection{Notation}
In addition to the notation from \eqref{def.eNmN}, we will use the $L^2$ spaces
\begin{equation*}
\| h\|_2
\eqdef
\sqrt{\int_{\mathbb{T}^3 } ~ dx ~ 
\int_{\mathbb{R}^3} ~ dp ~  
|h(x,p)|^2},
\quad
| h|_2
\eqdef
\sqrt{\int_{\mathbb{R}^3} ~ dp ~  
|h(p)|^2}.
\end{equation*}
Similarly any norm represented by one set of lines instead of two only takes into account the momentum variables.
We also define 
$
\| h\|_\nu
=
\| h\|_2,
$
which is justified by \eqref{nuDEF} later on.  The $L^2(\mathbb{R}^3_p)$ inner product is denoted $\langle \cdot, \cdot \rangle$.   We use $(\cdot, \cdot )$ to denote the  $L^2(\mathbb{T}^3_x \times \mathbb{R}^3_p)$ inner product.
Now, for $\ell\in\mathbb{R}$, we consider the weighted spaces
\begin{gather*}
\|  h \|_{2,\ell}
\eqdef
\| w_{\ell} h \|_{2},
\quad
\| h\|_{\nu,\ell}
\eqdef
\| w_{\ell} h\|_{\nu},
\quad
|  h |_{2,\ell}
\eqdef
| w_{\ell} h |_{2},
\quad
| h|_{\nu,\ell}
\eqdef
| w_{\ell} h |_{\nu}.
\end{gather*}
We will furthermore use $A \lesssim B$ to mean
that $\exists C>0$ such that $A \leq C B$ holds uniformly over the range of
parameters which are present in the inequality (and that the precise
magnitude of the constant is unimportant). 
%In particular, whenever either $A$ or $B$ involves a function space norm, it will be implicit that the constant is uniform over all elements of the relevant space unless explicitly stated otherwise.  
The notation $B \gtrsim A$ is equivalent to $A \lesssim B$, and $A \approx B$
means that both $A \lesssim B$ and $B \lesssim A$.  We additionally use $C>0$ to denote a generic positive large constant and $c>0$ to denote a small constant; their exact values are considered to be inessential.

\subsection{Organization of the paper}  In Section \ref{sec:momD} we will prove the estimates for the momentum derivatives of the non-linear collision operator \eqref{gamma0} using the splitting into the Glassey-Strauss frame and the center of momentum frame.  Then in Section \ref{sec:linearE} we will use these nonlinear estimates to deduce quickly several linear estimates, using also \cite{strainSOFT}.  Lastly in Section \ref{sec2}, we show how to use our estimates to prove the global existence and rapid decay, following the arguments from \cite{MR2100057,MR2209761}.

\section{Momentum derivatives of the nonlinear collision operator}\label{sec:momD}

Recalling the decomposition of $A$ and $A_{c}$ in (\ref{ABsplitting})$,$
consider the smooth test function $\chi \in C_{0}^{\infty }([0,\infty ))$
such that $0\leq \chi \leq 1$, and $\chi (\rho )=1$ for $\rho \in \lbrack
0,1]$ with $\chi (\rho )=0$ for $\rho >2$. We use the splitting $1=\chi
_{A}(p,q)+\chi _{A_{c}}(p,q)$ with 
\begin{eqnarray*}
&&\chi _{A}(p,q)\eqdef \chi \left( \pZ \right)
+\left( 1-\chi \left( \pZ \right) \right) \chi \left( \frac{|p|^{\frac{1}{m}%
}}{\qZ}\right) , \\
&&\chi _{A_{c}}(p,q)\eqdef \left( 1-\chi \left(
\pZ \right) \right) \left( 1-\chi \left( \frac{|p|^{\frac{1}{m}}}{\qZ}%
\right) \right) .
\end{eqnarray*}%
We split $\Gamma (f_{1},f_{2})=\Gamma _{A}+\Gamma _{A_{c}}$ as 
\begin{equation}
\begin{split}
\Gamma _{A} &=
\int_{{\mathbb{R}}^{\dim }\times {\mathbb{S}}^{2}}  d\omega dq ~ \frac{s B(p,q,\omega) }{\pZ \qZ}
\sqrt{J(q)}~
[f_{1}(\rP)f_{2}(\rQ)-f_{1}(p)f_{2}(q)]\chi _{A}(p,q), 
\\
\Gamma _{A_{c}} &=
\int_{{\mathbb{R}}^{\dim }\times {\mathbb{S}}^{2}} d\omega dq ~
v_{\o}
\sqrt{J(q)}~
[f_{1}(\PrP)f_{2}(\QrQ)-f_{1}(p)f_{2}(q)]\chi_{A_{c}}(p,q).
\end{split}
\label{splitG}
\end{equation}
Here without loss of generality, we have taken $f_1$ and $f_2$ to be scalar functions.
Using these important decompositions, we will prove the main estimate:

\begin{theorem}\label{thm:nonLINest}
We have the following nonlinear estimate for any $|\beta| \ge 0$:
$$
\left| \ang{w^{2}_{\ell} \partial_\beta \Gamma(f_1,f_2),\partial_\beta f_3} \right|
\lesssim
| \partial_{\beta} f_3 |_{2,\ell}
\sum_{\beta_1 + \beta_2 \le \beta}
| \partial_{\beta_1} f_1 |_{2,\ell}
| \partial_{\beta_2} f_2 |_{2,\ell}.
$$
Here we can include any $\ell \ge 0$.  Then for $|\ag| + |\be| \le N$ with $N\ge 4$ we have
$$
\left| \left( w^{2}_{\ell} \partial_\be^\ag \Gamma(f_1,f_2),\partial_\be^\ag f_3\right) \right|
\lesssim
\| \partial_\be^\ag f_3 \|_{2,\ell}
\prod_{j=1,2}
\sum_{|\ag_1| + |\be_1| \le N}
\| \partial_{\be_1}^{\ag_1} f_j \|_{2,\ell}.
$$
The second estimate follows easily from the first and Sobolev embeddings.
\end{theorem}

Now this theorem will follow directly from our Lemmas \ref{thm:nonLINestA} and \ref{nonlinGest:B} below.  It will be our focus in the rest of this section to prove these estimates.

\subsection{Estimates in the Glassey-Strauss Frame $\Gamma _{A}$}\label{sec:gsA}

To avoid taking derivatives for the singular factor of $|\omega \cdot (\frac{p}{\pZ }-\frac{q}{\qZ})|$ inside $B(p,q,\omega )$ for $\partial _{\beta}\Gamma _{A}$ in \eqref{splitG}, we introduce the following change of variables $q\rightarrow u
$ (for fixed $p)$ as: 
\begin{equation}
u=\pZ q-\qZ p.  \label{uCHANGE}
\end{equation}
By (\ref{uCHANGE}), we have that $q=\frac{\qZ}{\pZ }p+\frac{u}{\pZ }$ and taking
norms on both sides yields 
\begin{equation*}
\qZ=(u\cdot p)+\sqrt{(u\cdot p)^{2}+|u|^{2}+(\pZ )^{2}}.
\end{equation*}%
Such a transformation (\ref{uCHANGE}) therefore defines an invertible
mapping with 
\begin{eqnarray*}
\frac{\partial u_{i}}{\partial q_{j}} &=&\pZ \delta _{ij}-\frac{q_{j}p_{i}}{%
\qZ}, \quad (i,j = 1,2,3),
\\
\left\vert \frac{\partial u}{\partial q}\right\vert  &=&\det \left( \frac{%
\partial u_{i}}{\partial q_{j}}\right) =\frac{(\pZ )^{2}}{\qZ}\left(
\pZ \qZ-p\cdot q\right) \geq \frac{(\pZ )^{2}}{\qZ}.
\end{eqnarray*}%
Since $|\omega \cdot (\frac{p}{\pZ }-\frac{q}{\qZ})|=\frac{\left\vert
\omega \cdot u\right\vert }{\pZ \qZ},$ we can express $\Gamma _{A}$ from \eqref{splitG} as 
\begin{equation}
\Gamma_{A}
=
\int_{{\mathbb{R}}^{\dim }\times {\mathbb{S}}^{2}}d\omega du
\left\vert \frac{\partial q}{\partial u}\right\vert \frac{s\tilde{B}\left\vert \omega \cdot u\right\vert }{\pZ \qZ}
\sqrt{J(q)}\{f_{1}(\rP)f_{2}(\rQ)-f_{1}(p)f_{2}(q)\}\chi _{A}(p,q),
\label{gainTchange}
\end{equation}%
where now 
\begin{equation}
\tilde{B}\eqdef \frac{(\pZ +\qZ)^{2}}{\left[ (\pZ +\qZ)^{2}-(\omega \cdot
\lbrack p+q])^{2}\right] ^{2}}.  \notag
\end{equation}
We take a high order derivative $\partial _{\beta }$ of \eqref{gainTchange}
to obtain 
\begin{eqnarray*}
|\partial _{\beta }\Gamma _{A}| 
&\lesssim &
\sum
\int_{{\mathbb{R}}^{\dim }\times {\mathbb{S}}^{2}}~\mathbf{1}%
_{|p|^{\frac{1}{m}}\lesssim \qZ}d\omega du~K_{\beta _{0}}^{A}~\sqrt{J(q)}%
~|(\partial _{\beta _{1}}f_{1})(\rP)(\partial _{\beta
_{2}}f_{2})(\rQ)~\mu _{\beta _{2}}^{\beta _{1}}| \\
&&
+
\sum
\int_{{\mathbb{R}}^{\dim }\times {\mathbb{S}}^{2}}~\mathbf{1}_{|p|^{\frac{1}{m}}\lesssim
\qZ}d\omega du~K_{\beta _{0}}^{A}~\sqrt{J(q)}~|\partial _{\beta
_{1}}f_{1}(p)\partial _{\beta _{2}}f_{2}(q)|,
\end{eqnarray*}%
where the sum is over ${\beta _{0}+\beta _{1}+\beta_{2}\leq \beta }$.  Furthermore
\begin{equation}
K_{\beta _{0}}^{A}=K_{\beta _{0}}^{A}(u,p,\omega )\overset{\mbox{\tiny{def}}}%
{=}\left\vert \omega \cdot u\right\vert \left\vert \partial _{\beta
_{0}}\left( \left\vert \frac{\partial q}{\partial u}\right\vert ~\frac{s%
\tilde{B}}{\pZ \qZ}~J^{1/2}(q)\chi _{A}(p,q)\right) \right\vert
J^{-1/2}(q).  \label{KestNEED}
\end{equation}%
Also $\mu _{\beta _{2}}^{\beta _{1}}=\mu _{\beta _{2}}^{\beta
_{1}}(u,p,\omega )$ is the term which results from applying the chain rule
to the post-collisional velocities $\rP$ and $\rQ$. Here $\mu _{\beta _{2}}^{\beta _{1}}$ contains the sum of products of high order momentum derivatives of the smooth functions $\rP$ and $\rQ$. The next step is to reverse this change of variables \eqref{uCHANGE} to go from $u$ back to an integration over $q$. After that change of variables:

\begin{lemma}
\label{lem:GSderivativeEST} On the set $A,$ we have the following estimates 
\begin{equation*}
w_{\ell}^{2}(p)\left\vert K_{\beta _{0}}^{A}(\pZ q-\qZ p,p,\omega
)\right\vert \lesssim \frac{\left\langle {q}\right\rangle ^{n}}{\pZ \qZ}.
\end{equation*}%
Similarly, we also have the upper bound of 
\begin{equation*}
\left\vert \mu _{\beta _{2}}^{\beta _{1}}(\pZ q-\qZ p,p,\omega
)\right\vert \lesssim \left\langle {q}\right\rangle ^{n}.
\end{equation*}%
Above $n\geq 1$ is a fixed large integer which depends upon $\ell \geq 0$, 
$\beta $, $\beta _{0}$, $\beta _{1}$, and $\beta _{2}$. 
\end{lemma}

\begin{proof}[Proof of Lemma \protect\ref{lem:GSderivativeEST}]
We start with the estimate for $\mu _{\beta _{2}}^{\beta _{1}}$.  Clearly, up
to constants, $\mu _{\beta _{2}}^{\beta _{1}}$ is a sum of products of terms
of the following form 
\begin{equation*}
(\partial _{\beta _{1}}\rP)^{\gamma _{1}}(\partial _{\beta
_{2}}\rQ)^{\gamma _{2}},
\end{equation*}%
where $\beta _{1}$, $\beta _{2}$, $\gamma _{1}$ and $\gamma _{2}$ are
suitable multi-indices which are all $\le |\beta |$.
(Note $\beta _{1}$ and $\beta _{2}$ in the previous display need not be the same as those in $\mu _{\beta _{2}}^{\beta _{1}}$.)
 It is therefore sufficient to estimate the size of
these derivatives from above. This follows from the multi-dimensional
generalization \cite{MR1325915}, from 1996, of the \textbf{Fa{\`{a}} di
Bruno formula} (1855).   Consider the case $|\beta| = 1$.  After the change of variables \eqref{uCHANGE} we have 
\begin{eqnarray}
\rP &=&p+\tilde{a}(p,u,\omega )\omega ,  \notag \\
\rQ &=&\frac{u}{\pZ }+\frac{\qZ }{\pZ }p-\tilde{a}(p,u,\omega
)\omega ,  \notag
\end{eqnarray}%
where 
\begin{equation*}
\tilde{a}(p,u,\omega )=\frac{2(\pZ +\qZ )~\omega \cdot u}{%
(\pZ +\qZ )^{2}-\left\{ \omega \cdot \left( p+q\right) \right\} ^{2}}=%
\frac{N}{D}.
\end{equation*}%
Our goal will be to estimate derivatives of these functions. We thus compute 
\begin{eqnarray}
\frac{\partial p_{k}^{\prime }}{\partial p_{j}} &=&\delta _{kj}+\omega _{k}%
\frac{\partial \tilde{a}(p,u,\omega )}{\partial p_{j}},  \notag \\
\frac{\partial q_{k}^{\prime }}{\partial p_{j}} &=&\frac{u_{k}p_{j}}{%
(\pZ )^{3}}+\frac{\qZ }{\pZ }\delta _{kj}+\frac{\qZ }{(\pZ )^{3}}%
p_{k}p_{j}+\frac{\partial \qZ }{\partial p_{j}}\frac{p_{k}}{\pZ }+\omega
_{k}\frac{\partial \tilde{a}(p,u,\omega )}{\partial p_{j}}.  \notag
\end{eqnarray}%
We compute the final derivative as 
\begin{gather*}
\frac{\partial \tilde{a}}{\partial p_{j}}=\frac{1}{D}\frac{\partial N}{%
\partial p_{j}}-\frac{N}{D^{2}}\frac{\partial D}{\partial p_{j}}, \\
\frac{\partial N}{\partial p_{j}}=2\omega \cdot u\left( \frac{p_{j}}{\pZ }+%
\frac{\partial \qZ }{\partial p_{j}}\right) , \\
\frac{\partial \qZ }{\partial p_{j}}=u_{j}+\frac{2(u\cdot p)u_{j}+2p_{j}}{%
\sqrt{(u\cdot p)^{2}+|u|^{2}+(\pZ )^{2}}}.
\end{gather*}%
The derivative of the denominator is further given by 
\begin{multline*}
\frac{\partial D}{\partial p_{j}}=2\left( \pZ +\qZ \right) \left( \frac{%
p_{j}}{\pZ }+\frac{\partial \qZ }{\partial p_{j}}\right) 
 \\
-2\omega \cdot \left( p+q\right) 
\sum_{k=1}^3\omega _{k}\left( \delta _{kj}+%
\frac{\qZ }{\pZ }\delta _{kj}+p_{k}\frac{\partial }{\partial p_{j}}\left( 
\frac{\qZ }{\pZ }\right) +\frac{u_{k}p_{j}}{(\pZ )^{3}}\right) .
\end{multline*}%
In particular we can write the whole derivatives of ($p^\prime$, $q^\prime$) as
\begin{eqnarray}
\frac{\partial p_{k}^{\prime }}{\partial p_{j}} &=&
\frac{f_{kj}(p,q,\omega, \pZ, \qZ, u)}{(\pZ)^a D^b (\sqrt{(u\cdot
p)^{2}+|u|^{2}+(\pZ )^{2}})^c},
\notag 
\\
\frac{\partial q_{k}^{\prime }}{\partial p_{j}} &=&
\frac{g_{kj}(p,q,\omega, \pZ, \qZ, u)}{(\pZ)^a D^b (\sqrt{(u\cdot
p)^{2}+|u|^{2}+(\pZ )^{2}})^c},
\notag
\end{eqnarray}
where $f_{kj}$ and $g_{kj}$ are smooth polynomials in the variables $(p,q,\omega, \pZ, \qZ, u)$.  Furthermore, 
$a$, $b$ and $c$ are positive exponents which depend upon $k$, $j$ and the form of ($p^\prime$, $q^\prime$).  These polynomials and exponents are quite lengthy to compute.  However, the key observation in these calculations is that the
denominators, e.g. $\pZ $, $D$, and $\sqrt{(u\cdot
p)^{2}+|u|^{2}+(\pZ )^{2}}$ are in all cases uniformly bounded from below,
so that no singularities are present. 

Therefore, after applying the reverse
change of variables $u\rightarrow \pZ q-\qZ p$ to \eqref{uCHANGE} we can
see that we always have the crude upper bound (for some $n>0$) of 
\begin{equation*}
\left\vert \frac{\partial p_{k}^{\prime }}{\partial p_{j}}\right\vert
+\left\vert \frac{\partial q_{k}^{\prime }}{\partial p_{j}}\right\vert
\lesssim (\left\langle {p}\right\rangle \left\langle {q}\right\rangle )^{n}.
\end{equation*}%
This estimate will conclude the second estimate in Lemma \ref%
{lem:GSderivativeEST} if all of the derivatives are first order derivatives.
Note that the exact value of $n$ is in fact unimportant to our argument. The
crucial observation now is that this pattern repeats for the higher
derivatives of order $\beta$ with $|\beta| >1$.  In particular, we see that
\begin{eqnarray}
%\frac{\partial^\beta p_{k}^{\prime }}{\partial p_\beta} 
\partial_\beta p_{k}^{\prime}
&=&
\frac{f_{\beta j}(p,q,\omega, \pZ, \qZ, u)}{(\pZ)^a D^b (\sqrt{(u\cdot
p)^{2}+|u|^{2}+(\pZ )^{2}})^c},
\notag 
\\
%\frac{\partial^\beta q_{k}^{\prime }}{\partial p_\beta} 
\partial_\beta q_{k}^{\prime}
&=&
\frac{g_{\beta j}(p,q,\omega, \pZ, \qZ, u)}{(\pZ)^a D^b (\sqrt{(u\cdot
p)^{2}+|u|^{2}+(\pZ )^{2}})^c},
\notag
\end{eqnarray}
where again $f_{\beta j}$ and $g_{\beta j}$ are smooth polynomials in the variables $(p,q,\omega, \pZ, \qZ, u)$.   And once again
$a$, $b$ and $c$ are (different) positive exponents which will depend upon $\beta$, $k$ and the form of ($p^\prime$, $q^\prime$).  This form of these high order derivatives is quickly deduced from the standard rules of differentiation, and for instance a simple induction procedure.  However to compute the exact expressions of $f_{\beta j}$ and $g_{\beta j}$ seems to be quite difficult.  A key observation is that computing these high order polynomial expressions explicitly is in fact not-necessary to our argument.  

Again the crucial point is that the denominators of 
$\partial_\beta p_{k}^{\prime}$
%$\frac{\partial^\beta p_{k}^{\prime }}{\partial p_\beta}$
and
$\partial_\beta q_{k}^{\prime}$
%$\frac{\partial^\beta q_{k}^{\prime }}{\partial p_\beta}$ 
are uniformly bounded from below by a positive constant, so that no singularities are present.
This implies that, after applying the change of variables $u\rightarrow \pZ q-\qZ p$, the higher order derivatives are all similarly bounded
above as
\begin{equation*}
\left\vert 
\partial_\beta p_{k}^{\prime}
\right\vert
+
\left\vert 
\partial_\beta q_{k}^{\prime}
\right\vert
\lesssim (\left\langle {p}\right\rangle \left\langle {q}\right\rangle )^{n},
\end{equation*}
for some (different) $n>0$. After that, the second estimate for $\mu _{\beta
_{2}}^{\beta _{1}}$ in Lemma \ref{lem:GSderivativeEST} is a consequence of
the region $A$ from \eqref{ABsplitting}. We conclude the estimate for $\mu
_{\beta _{2}}^{\beta _{1}}$.

The estimate for \eqref{KestNEED} is directly similar. The point is again that
momentum derivatives in \eqref{KestNEED} end up creating ratios of polynomials in the
variables $\omega$, $u$, $p$, $\pZ $, and $\qZ$.  The exact expressions created by the high order derivative 
$\partial_{\beta_0}$ in \eqref{KestNEED} is apparently quite difficult to compute in general, but fortunately this is not necessary.    Instead it is easily seen that the
denominators of these rational functions created by the high order derivative, $\partial_{\beta_0}$,  are all uniformly bounded from below. Thus
after reversing the change of variables in \eqref{uCHANGE} the term %
\eqref{KestNEED} must be bounded from above by a constant multiple of $%
(\left\langle {p}\right\rangle \left\langle {q}\right\rangle )^{n}$ for some 
$n>0$.
\end{proof}

%We are now ready to establish the following:

\begin{lemma}
\label{thm:nonLINestA} We have the following estimate for $\Gamma _{A}$ with any $\ell \geq 0$:  
\begin{eqnarray}\notag
&
|w^{2}_{\ell}(p)\partial _{\beta }\Gamma _{A}|\lesssim &
e^{-c|p|^{1/m}}
\int_{{\mathbb{R}}^{\dim}\times {\mathbb{S}}^{2}}~ dq d\omega~\mathbf{1}_{|p|_{{}}^{\frac{1}{m}}\lesssim
\qZ }
~J^{\frac{1}{4}}(q)~
\\
&&
\times \sum_{\beta _{1}+\beta _{2}\leq \beta }
 \left\{
|(\partial _{\beta_{1}}f_{1})(\rP)(\partial _{\beta _{2}}f_{2})(\rQ)|  
+
|(\partial _{\beta _{1}}f_{1})(p)(\partial _{\beta
_{2}}f_{2})(q)|
\right\}.
 \label{gammaA}
\end{eqnarray}
Moreover, from that estimate one can deduce the following uniform bound
\begin{equation}
\left\vert \left\langle {w^{2}_{\ell }\partial _{\beta }\Gamma }_{A}{%
(f_{1},f_{2}),\partial _{\beta }f_{3}}\right\rangle \right\vert \lesssim
|\partial _{\beta }f_{3}|_{2}
\sum_{\beta _{1}+\beta _{2}\leq \beta}|\partial _{\beta _{1}}f_{1}|_{2 }
|\partial _{\beta_{2}}f_{2}|_{2 }.  \label{gammaA123}
\end{equation}
\end{lemma}

%\Red{FIX UP THE THEOREM ABOVE. ADD COMMENTARY HERE.}

\begin{proof} The proof of (\ref%
{gammaA}) follows directly from the previous Lemma \ref{lem:GSderivativeEST} by noting that $J^{-\frac{%
1}{8}}(q)\gtrsim e^{c|p|^{1/m}}\gtrsim ( 1+|p|+|q|)^{n}$ for any large $n>0$ 
since $|p| \lesssim (\qZ )^{m}$. 

To establish (\ref{gammaA123}), we apply the
Cauchy-Schwartz inequality to obtain
\begin{multline*}
\left\vert \left\langle {w^{2}_{\ell }\partial _{\beta }\Gamma _{A},\partial
_{\beta }f_{3}}\right\rangle \right\vert 
\\
\lesssim \sum_{\beta _{1}+\beta
_{2}\leq \beta }\int ~\frac{J^{1/4}(q)}{\pZ \qZ }
\left\{\left\vert \partial
_{\beta _{1}}f_{1}(\rP)\partial _{\beta _{2}}f_{2}(\rQ)|+|\partial _{\beta _{1}}f_{1}(p)\partial _{\beta
_{2}}f_{2}(q)|\right\}
|\partial _{\beta }f_{3}(p)\right\vert 
%d\omega dpdq 
\\
\lesssim |\partial _{\beta }f_{3}|_{2 }\sum_{\beta _{1}+\beta
_{2}\leq \beta }\int ~d\omega dqdp~\frac{J^{1/4}(q)}{\pZ \qZ }~\left\vert
(\partial _{\beta _{1}}f_{1})(\rP)(\partial _{\beta
_{2}}f_{2})(\rQ)\right\vert ^{2}
\\
+|\partial _{\beta }f_{3}|_{2 }\sum_{\beta _{1}+\beta _{2}\leq \beta
}|\partial _{\beta _{1}}f_{1}|_{2}|\partial _{\beta _{2}}f_{2}|_{2}.
\end{multline*}
We finally complete the proof by making pre-post collisional change of
variables $(p,q)\rightarrow (\rP,\rQ)$ from \eqref{PARTICLEjacobian}.
\end{proof}

\subsection{Center of Momentum Frame}

\label{sec:comV} In this section we prove estimates for the term $\Gamma _{A_{c}}$ from \eqref{splitG}.  We take $\beta$
momentum derivatives of $\Gamma _{A_{c}}$ to obtain 
\begin{gather*}
|\pa_{\be}\Gamma _{A_{c}}| 
\lesssim 
\sum
\int_{{\mathbb{R}}^{\dim }\times {\mathbb{S}}^{2}}
dqd\omega ~
\mathbf{1}_{|p|^{1/m}\gtrsim \qZ }  
|\partial _{\beta _{0}}v_{\o }|
\sqrt{J(q)}|(\partial _{\beta _{1}}f_{1})(\PrP)(\partial _{\beta
_{2}}f_{2})(\QrQ)\kappa _{\beta _{2}}^{\beta _{1}}| 
\\
+\sum
|\partial _{\beta_{1}}f_{1}(p)|\int_{{\mathbb{R}}^{\dim }\times {\mathbb{S}}^{2}} dqd\omega ~ 
\mathbf{1}_{|p|^{1/m}\gtrsim \qZ } ~|\partial _{\beta _{0}}v_{\o }|\sqrt{J(q)}\pa_{\be_2} f_{2}(q)|.
\end{gather*}
Here $\kappa _{\beta _{2}}^{\beta _{1}}$ is the collection of sums of products of momentum derivatives of $\PrP$ and $\QrQ$, from \eqref{pcM}, which result from the chain rule of differentiation.  Again the sum is over the multi-indices ${\beta _{0}+\beta_{1}+\beta _{2}\leq \beta }$.  
We then have

\begin{lemma}
\label{derivativeESTlem} Let $|p|^{\frac{1}{m}}\gtrsim \qZ$ with $m$ large, as in 
\eqref{ABsplitting}.  For some integer $n\geq 1$, which depends upon $\beta \neq 0$, we have the following estimates
\begin{equation}
\frac{\left\langle {p}\right\rangle |\partial _{\beta }v_{\o}|}{| v_{\o} |}+\left\vert \partial _{\beta }\PrP\right\vert +\left\vert
\partial _{\beta }\QrQ\right\vert \lesssim \left\langle {q}%
\right\rangle ^{n}.  \notag
\end{equation}
\end{lemma}

To prove Lemma \ref{derivativeESTlem} we will use the following:  

\begin{lemma}\label{newLEMgs}
Let $|p|^{\frac{1}{m}}\gtrsim q^{0}$ with $m$ large.
% Recall $g=\sqrt{%2(p^{0}q^{0}-p\cdot q-1)\text{ }}$ and $s=g^{2}+4.$ 
Then for any $\beta \neq
0$, we have 
\begin{eqnarray*}
|\partial _{\beta }g| &\lesssim &\frac{\langle q\rangle ^{|\beta |}}{g},%
\text{ \ \ \ \ } \left|\partial _{\beta }\left( \frac{1}{g}\right) \right|
\lesssim 
\frac{\langle q\rangle ^{|\beta |}}{g^{3}}, 
\\
|\partial _{\beta }\sqrt{s}| &\lesssim &\frac{\langle q\rangle ^{|\beta |}}{g%
},\text{ \ \ \ \ }
\left|\partial _{\beta }\left( \frac{1}{\sqrt{s}}\right) \right|
\lesssim \frac{\langle q\rangle ^{|\beta |}}{g^{3}}.
\end{eqnarray*}
\end{lemma}

\begin{proof}[Proof of Lemma \protect\ref{newLEMgs}]  
We shall use an induction over $|\beta |.$ If $|\beta |=1,$ we have 
\begin{eqnarray*}
|\partial _{p_{i}}g| &=&\left\vert \frac{1}{g}\left( \frac{p_{i}}{p^{0}}%
q^{0}-q_{i}\right) \right\vert \lesssim \frac{\ang{q}}{g}, \\
\left\vert \partial _{p_{i}}\left( \frac{1}{g}\right) \right\vert 
&=&\left\vert -\frac{\partial _{p_{i}}g}{g^{2}}\right\vert 
\lesssim \frac{\ang{q}}{g^{3}}, \\
|\partial _{p_{i}}\sqrt{s}| &=&\left\vert \frac{g\partial _{p_{i}}g}{\sqrt{s}%
}\right\vert \lesssim |\partial _{p_{i}}g|\lesssim \frac{\ang{q}}{g}, 
\\
\left\vert \partial _{p_{i}}\left( \frac{1}{\sqrt{s}}\right) \right\vert 
&=&
\left|\frac{g\partial _{p_{i}}g}{\sqrt{s^{3}}}\right|\lesssim 
\frac{|\partial_{p_{i}}g|}{g^{2}}\lesssim \frac{\ang{q}}{g^{3}}.
\end{eqnarray*}
Here we recall that $s=g^2+4$.
Assume that the estimates are valid for $\beta .$ Now for  $(|\beta |+1)-$th
order derivatives, from the induction hypothesis for $\frac{1}{g},$ we have 
\begin{eqnarray*}
|\partial _{\beta }\partial _{p_{i}}g| 
&\lesssim &
\sum_{\beta _{1}\leq \beta}\left\vert \partial _{\beta _{1}}\left( \frac{1}{g}\right)
\partial _{\beta -\beta _{1}}\left( \frac{p_{i}}{p^{0}}q^{0}-q_{i}\right)
\right\vert  \\
&\lesssim &\frac{\langle q\rangle }{g}+\sum_{0\neq \beta _{1}\leq \beta }%
\frac{\langle q\rangle ^{|\beta _{1}|}}{g^{3}}\langle q\rangle  \\
&\lesssim &\frac{\langle q\rangle ^{|\beta |+1}}{g}.
\end{eqnarray*}%
Here (and below) we use that $g \gtrsim 1$ on $|p|^{\frac{1}{m}}\gtrsim q^{0}$.  This follows from
\begin{equation}\label{boundSg}
1 \lesssim \sqrt{\frac{\left\langle {p}\right\rangle }{\left\langle {q}\right\rangle }}%
\lesssim g\lesssim \left\langle {p}\right\rangle,
\end{equation}
which itself is a consequence of
the inequality $\frac{|p-q|}{\sqrt{\pZ \qZ }%
}\lesssim \relMOM\lesssim |p-q|$ (see \cite[Lemma 3.1]{MR1211782}) on the region $|p|^{\frac{1}{m}}\gtrsim \qZ$ from %
\eqref{ABsplitting}.
For the next step, similarly note that 
\begin{eqnarray*}
\left\vert \partial _{\beta }\partial _{p_{i}}\left\{ \frac{1}{g}\right\}
\right\vert  
&=&\left\vert \partial _{\beta }\left\{ \frac{\partial _{p_{i}}g}{g^{2}}%
\right\} \right\vert  \\
&\lesssim &\sum_{\beta _{1}+\beta _{2}\leq \beta }\left\vert \partial
_{\beta -\beta _{1}-\beta _{2}}\left\{ \frac{1}{g}\right\} \partial _{\beta
_{1}}\left\{ \frac{1}{g}\right\} \partial _{\beta _{2}}\left\{ \partial
_{p_{i}}g\right\} \right\vert  \\
%&&+\sum_{\beta _{1}+\beta _{2}\leq \beta }\left\vert \partial _{\beta -\beta
%_{1}-\beta _{2}}\left\{ \frac{1}{g}\right\} \partial _{\beta _{1}}\left\{ 
%\frac{1}{g}\right\} \partial _{\beta _{2}}\left\{ \partial _{p_{i}}g\right\}
%\right\vert  \\
&\lesssim &
\sum_{\beta _{1}+\beta _{2}\leq \beta}
\frac{\langle
q\rangle ^{|\beta |-|\beta _{1}|-|\beta _{2}|}}{g}\frac{\langle q\rangle^{|\beta _{1}|}}{g}\frac{\langle q\rangle ^{|\beta _{2}|+1}}{g}
\\
&\lesssim &\frac{\langle q\rangle ^{|\beta |+1}}{g^{3}}.
\end{eqnarray*}
This last estimate in particular holds because we have, by the induction assumption and a direct calculation, for any multi-index $0 \le \alpha \le \beta$ that
$$
\left|\partial _{\alpha }\left\{ \frac{1}{g}\right\} \right|
\lesssim \frac{\langle
q\rangle^{|\alpha|} }{g}
$$
Here again we used that $g \gtrsim 1$ on $|p|^{\frac{1}{m}}\gtrsim q^{0}$.

Simialrly for any multi-index satisfying $0 \le \alpha \le \beta$ 
we have that
$$
\left\vert \partial _{\alpha }\left\{ \frac{1}{\sqrt{s}}\right\} \right\vert
\lesssim \frac{\langle q\rangle ^{|\alpha |}}{g},
\quad
\left\vert \partial _{\alpha } g \right\vert
\lesssim 
\max\left\{ g, \frac{\langle q\rangle ^{|\alpha |}}{g} \right\}.
$$
With that, we again use the induction hypothesis to obtain
\begin{eqnarray*}
|\partial _{\beta }\partial _{p_{i}}\sqrt{s}| &=&\left\vert \partial _{\beta
}\left\{ \frac{g\partial _{p_{i}}g}{\sqrt{s}}\right\} \right\vert  
\\
&\lesssim &\sum_{\beta _{1}+\beta _{2}\le \beta }\left\vert \partial
_{\beta -\beta _{1}-\beta _{2}}g\partial _{\beta _{1}}\left\{ \frac{1}{\sqrt{%
s}}\right\} \partial _{\beta _{2}}\partial _{p_{i}}g\right\vert  
\\
&\lesssim &
\sum_{\beta _{1}+\beta _{2}\le \beta } 
\max\left\{ g, \frac{\langle q\rangle ^{|\beta |-|\beta _{1}|-|\beta _{2}|}}{g} \right\}
\frac{\langle q\rangle^{|\beta _{1}|}}{g}
\frac{\langle q\rangle ^{|\beta _{2}|+1}}{g}
\\
&\lesssim &\frac{\langle q\rangle ^{|\beta |+1}}{g}.
\end{eqnarray*}%
For the last case, we do a similar calculation as
\begin{multline*}
\left\vert \partial _{\beta }\partial _{p_{i}}\left\{ \frac{1}{\sqrt{s}}%
\right\} \right\vert  
=
\left|\partial _{\beta }\left\{ \frac{g\partial _{p_{i}}g}{\sqrt{s^{3}}}\right\} \right| 
\\
\lesssim \sum_{\beta _{1}+\beta _{2}+\beta _{3}+\beta _{4}\leq \beta
}
\left| \partial _{\beta _{1}}\left\{ \frac{1}{\sqrt{s}}\right\} \partial _{\beta
_{2}}\left\{ \frac{1}{\sqrt{s}}\right\} \partial _{\beta _{3}}\left\{ \frac{1%
}{\sqrt{s}}\right\} \partial _{\beta _{4}}g\partial _{\beta -\beta
_{1}-\beta _{2}-\beta _{3}-\beta _{4}}\partial _{p_{i}}g 
\right|
\\
\lesssim 
\sum_{\beta _{1}+\beta _{2}+\beta _{3}+\beta _{4}\leq \beta }%
\frac{\langle q\rangle ^{|\beta _{1}|}}{g}
\frac{\langle q\rangle ^{|\beta_{2}|}}{g}
\frac{\langle q\rangle ^{|\beta _{3}|}}{g}
\max\left\{ g,\frac{\langle q\rangle^{|\beta _{4}|}}{g}\right\}
\frac{\langle q\rangle ^{|\beta |+1-|\beta _{1}|-|\beta_{2}|-|\beta _{3}|-|\beta _{4}|}}{g} 
\\
\lesssim \frac{\langle q\rangle ^{|\beta |+1}}{g^{3}}.
\end{multline*}
In summary, the desired estimates follow by via the induction hypothesis.
\end{proof}

We now use the estimates from Lemma \ref{newLEMgs} to prove Lemma \protect\ref{derivativeESTlem}. 

\begin{proof}[Proof of Lemma \protect\ref{derivativeESTlem}]  
We first show the decay of $\partial _{\beta}v_{\o}$. 
Consider $\relMOM$, and recall \eqref{boundSg}.   Next recall $v_{\o}=\frac{1}{2}\frac{g\sqrt{s}}{\pZ \qZ }$ from \eqref{moller}. Then using Lemma \ref{newLEMgs} we have
\begin{eqnarray*}
|\partial _{\beta }v_{\o}| &=& \frac{1}{2} 
\left|\partial _{\beta }\left\{ \frac{g\sqrt{s}}{%
\pZ \qZ }\right\} \right| 
\lesssim 
\sum_{\beta_1+\beta_2 \le \beta} 
\left|\partial _{\beta _{1}}g\partial _{\beta _{2}}\sqrt{s}%
\partial _{\beta - \beta_1 -\beta _{2}}\left(\frac{1}{\pZ \qZ }\right) \right| 
\\
&\lesssim &
\max\left\{ g, \frac{\langle q\rangle ^{|\beta_1 |}}{g} \right\}
\frac{\left\langle {q}\right\rangle ^{|\beta_2| }}{g}
\frac{1}{\pZ \qZ }%
\lesssim \frac{\left\langle {q}\right\rangle ^{|\beta|}}{g^{2}}\frac{g\sqrt{s}}{%
\pZ \qZ }\lesssim \frac{\left\langle {q}\right\rangle ^{|\beta|}}{\left\langle {p%
}\right\rangle } \left| v_{\o} \right|.
\end{eqnarray*}
This completes the estimate for a high-order derivative of $v_{\o}$.

To show $\left\vert \partial _{\beta }\PrP \right\vert +\left\vert \partial _{\beta }\QrQ\right\vert
\lesssim \left\langle {q}\right\rangle ^{n},$ we note from 
\eqref{pcM} for $|\beta |>0$ that
\begin{equation*}
|\partial _{\beta }p_{i}^{\prime \prime }|\lesssim 
\frac{\delta_{ij}}{2}\mathbf{1}_{\beta =e_{j}} 
+
\left\vert \frac{\partial _{\beta }g}{2}\omega
_{i}\right\vert 
+
\sum_{\alpha\le \be} \left\vert \partial _{\beta -\alpha }\left( \frac{g}{2}%
(\rho -1)\right) \partial _{\alpha }\left( (p_{i}+q_{i})\frac{(p+q)\cdot
\omega }{|p+q|^{2}}\right) \right\vert .
\end{equation*}%
We estimate each of these terms individually. The first term is
trivially bounded.   By Lemma \ref{newLEMgs}, for the second term on the right side of $\partial _{\beta }p_{i}^{\prime \prime }$ we use 
$|\partial _{\beta }g| \lesssim \frac{\langle q\rangle ^{|\beta |}}{g} \lesssim \langle q\rangle ^{|\beta |}$ on 
$|p|^{\frac{1}{m}}\gtrsim \qZ$. 
For the third and last term on the right side of $\partial _{\beta }p_{i}^{\prime \prime }$ we notice 
\begin{equation*}
\left\vert \partial _{\alpha }\left( (p_{i}+q_{i})\frac{(p+q)\cdot \omega }{%
|p+q|^{2}}\right) \right\vert \lesssim |p+q|^{-|\alpha |}\lesssim
\left\langle {p}\right\rangle ^{-|\alpha |}.
\end{equation*}%
The first inequality holds generally; the second inequality holds on 
$|p|^{\frac{1}{m}}\gtrsim \qZ$ from \eqref{ABsplitting}. Now for the term 
$\frac{g}{2}(\rho -1)$ we notice that $\left\vert \frac{g}{2}(\rho -1)\right\vert
\lesssim \pZ +\qZ \lesssim \left\langle {p}\right\rangle $ on $|p|^{\frac{1%
}{m}}\gtrsim \qZ$. The first inequality in the previous chain holds because of $%
\frac{g}{\sqrt{s}}\lesssim 1$ and the definition of $\rho $ from \eqref{pcM}. We conclude that if $\beta -\alpha =0$
then the third term on the right side of $\partial _{\beta }p_{i}^{\prime
\prime }$ is bounded as in Lemma \ref{derivativeESTlem} since in this case $%
|\alpha |=|\beta |>0$.

It remains to estimate the last term on the right side of $\partial _{\beta
}p_{i}^{\prime }$ when $|\beta -\alpha |>0$. To this end, notice that 
$
\frac{g}{2}(\rho -1)=\frac{g}{2}\left( \frac{\pZ +\qZ }{\sqrt{s}}-1\right).
$
Therefore by Lemma \ref{newLEMgs}
\begin{multline*}
\left| \partial _{\beta}\left( \frac{g}{2}(\rho -1)\right) \right|
\lesssim
\left| 
\left( \partial _{\beta}g\right) \left( \rho -1\right) \right|
+
\sum_{\beta_1 \ne \beta, \beta_1 + \beta_2 < \beta} 
\left| 
\left(\partial _{\beta_1} g\right) \partial _{\beta_2}\left( \frac{1}{\sqrt{s}} \right) 
\partial _{\beta - \beta_1 - \beta_2} \pZ \right|
\\
+
\sum_{\beta_1 \ne \beta, \beta_1 + \beta_2 = \beta} 
\left| 
\left(\partial _{\beta_1} g\right) \partial _{\beta_2}\left( \frac{1}{\sqrt{s}} \right) 
\partial _{\beta - \beta_1 - \beta_2} \pZ \right|
\\
\lesssim
\frac{\ang{q}^{|\beta|}}{g} \left| \frac{\pZ + \qZ -\sqrt{s}}{\sqrt{s}} \right|
+
\sum_{\beta_1 \ne \beta, \beta_1 + \beta_2 < \beta} 
\max\left\{ g, \frac{\ang{q}^{|\beta_1|}}{g}\right\} 
\frac{\ang{q}^{|\beta_2|}}{g}
\\
+
\sum_{\beta_1 \ne \beta, \beta_1 + \beta_2 = \beta} 
\max\left\{ g, \frac{\ang{q}^{|\beta_1|}}{g}\right\} 
\frac{\ang{q}^{|\beta_2|}}{g^3}  p^0
\\
\lesssim
\ang{q}^{|\beta|+2}.
\end{multline*}%
This use several previous estimates, and holds for a general multi-index $\beta \ne 0$.  

Collecting all of these completes the estimate for $\partial _{\beta}p_i^{\prime\prime}$ in Lemma \ref{derivativeESTlem}. Notice that the estimate for $\left| \partial _{\beta } q_i^{\prime\prime}\right|$ is exactly the same.
\end{proof}

%\Red{ADD COMMENTARY HERE.}

\begin{lemma}
\label{nonlinGest:B} Fix $|\beta |\ge 0$. Then we have the uniform estimate
\begin{multline}
| \partial _{\beta }\Gamma _{A_{c}}| 
%\\
\lesssim 
\sum 
\int_{{\mathbb{R}}^{\dim }\times {\mathbb{S}}^{2}}  dqd\omega ~
\mathbf{1}_{|p|^{1/m}\gtrsim
\qZ }
\frac{ v_{\o} ~ J^{1/4}(q)}{|p|^{\min \{1,|\beta_{0}|\}}}
|(\partial _{\beta _{1}}f_{1})(\PrP)(\partial _{\beta _{2}}f_{2})(\QrQ)| \\
+
\sum
|\partial _{\beta_{1}}f_{1}(p)|
\int_{{\mathbb{R}}^{\dim }\times {\mathbb{S}}^{2}} dqd\omega ~ \mathbf{1}_{|p|^{1/m}\gtrsim \qZ }
\frac{ v_{\o} ~ J^{1/4}(q)}{|p|^{\min \{1,|\beta_{0}|\}}}
| \partial _{\beta _{2}} f_{2}(q)|.
\label{estAc}
\end{multline}
Above the sum is over multi-indices ${\beta _{0}+\beta _{1}+\beta _{2}\leq \beta }$.  

Moreover, for $\ell \geq 0$, we obtain
\begin{equation}
\left\vert \left\langle {w^{2}_{\ell }\partial _{\beta }\Gamma
_{A_{c}}(f_{1},f_{2}),\partial _{\beta }f_{3}}\right\rangle \right\vert
\lesssim |\partial _{\beta }f_{3}|_{2,\ell }\sum_{\beta _{1}+\beta _{2}\leq
\beta }|\partial _{\beta _{1}}f_{1}|_{2,\ell }|\partial _{\beta
_{2}}f_{2}|_{2,\ell }.
\label{nonLest}
\end{equation}
\end{lemma}

%\Red{ADD COMMENTARY HERE.}

\begin{proof}
First \eqref{estAc} follows directly from the previous Lemma \ref{derivativeESTlem} and the fact that  $J^{-1/4}(q)\gtrsim \ang{q}^{n}$ for any $n>0$. We then have
the upper bound of
\begin{gather*}
\left\vert \ang{ {w^{2}_{\ell }\partial _{\beta}\Gamma }_{A_{c}}{,\partial _{\beta }f_{3}} } \right\vert 
\lesssim 
\sum\int  %dqdpd\omega 
w^{2}_{\ell }(p)   v_{\o}J^{1/4}(q)~\left\vert (\partial _{\beta _{1}}f_{1})(\PrP)(\partial _{\beta _{2}}f_{2})(\QrQ)\partial _{\beta
}f_{3}(p)\right\vert 
\\
+\sum
\int dqdpd\omega ~ w^{2}_{\ell }(p)  v_{\o}J^{1/4}(q) ~ |\partial_{\be_2} f_{2}(q)| ~ |\partial _{\beta _{1}}f_{1}(p)\partial _{\beta
}f_{3}(p)|.
\end{gather*}
Above the sum is over multi-indices ${\beta _{1}+\beta _{2}\leq \beta }$.
The second term above clearly has the desired upper bound in \eqref{nonLest} using Cauchy-Schwartz.  

For the first ``gain term'',
notice from \cite[Lemma 2.2]{MR1211782} that we have the estimate 
$w_{\ell}(p)\leq w_{\ell }(\PrP)w_{\ell }(\QrQ)$.   We remark that the estimate \cite[Lemma 2.2]{MR1211782} is true for any variables satisfying the conservation laws \eqref{collisionalCONSERVATION}.
Now using Cauchy-Schwartz we obtain the upper bound (using also $v_{\o}\leq 4$) 
\begin{multline*}
|\partial _{\beta }f_{3}|_{2,\ell }\left\{ \sum
\int ~v_{\o}~w^{2}_{\ell }(\PrP)w^{2}_{\ell}(\QrQ)\left\vert \partial _{\beta _{1}}f_{1}(\PrP)\partial _{\beta
_{2}}f_{2}(\QrQ)\right\vert ^{2}dpdq d\omega\right\} ^{1/2} 
\\
\lesssim 
|\partial _{\beta }f_{3}|_{2,\ell }\left\{ \sum
\int ~v_{\o}~w^{2}_{\ell}(\PrP)w^{2}_{\ell}(\QrQ)\left\vert \partial _{\beta _{1}}f_{1}(\PrP)\partial
_{\beta _{2}}f_{2}(\QrQ)\right\vert ^{2}d\PrP d\QrQ d\omega\right\} ^{1/2}.
\end{multline*}
Once again the pre-post change of variables from \eqref{changeCM}
establishes Lemma \ref{nonlinGest:B}.  
\end{proof}

%\Red{ADD DISCUSSION HERE?}

\section{The linear estimates}\label{sec:linearE}

We will use the estimates proven in the previous section for $\Gamma$ from \eqref{gamma0} to prove the linear estimates in this section. 
Recalling \eqref{L} and \eqref{gamma0} we write
\begin{equation}
L(h)=[L_+(h), L_-(h)], \quad L(h) = \nu(p) h - K(h).
\notag
\end{equation}
Here we recall that $h=[h_+, h_-]$.  From \eqref{rCOL}, \eqref{gamma0} and \eqref{collisionGS} we can define
\begin{equation}
\notag
\begin{split}
\Gamma^{gain}(f_1, f_2) \eqdef &
\int_{{\mathbb{R}}^{\dim}\times {\mathbb{S}}^{2}}  d\omega dq ~ \frac{s B(p,q,\omega)}{\pZ \qZ}
\sqrt{J(q)}
f_1(\rP)f_2(\rQ),  
\\
\Gamma^{loss}(f_1, f_2) \eqdef &
\int_{{\mathbb{R}}^{\dim}\times {\mathbb{S}}^{2}}  d\omega dq ~ \frac{s B(p,q,\omega)}{\pZ \qZ}
\sqrt{J(q)}
f_1(p)f_2(q).
\end{split}
\end{equation}
Above $f_1$ and $f_2$ are scalar functions.  
Then following \eqref{gamma0} we can write
\begin{equation}
\nu (p)
\eqdef
2\Gamma^{loss}\left(1, \sqrt{J}\right)
=
2\int_{{\mathbb{R}}^{\dim }\times {\mathbb{S}}^{2}} d\omega dq  ~ \frac{sB(p,q,\omega )}{\pZ \qZ }~J(q).
\label{nuDEF}
\end{equation}%
With these developments the operator $K(h) = [K_+(h), K_-(h)]$ can be expressed as
\begin{equation}
K_\pm(h)\eqdef 
\Gamma^{gain}(h_{\pm}, \sqrt{J}) 
+
\Gamma^{gain}(h_{\pm}, \sqrt{J})
+\Gamma_{\pm}\left( [\sqrt{J},\sqrt{J}], h\right).
\label{kOPdef}
\end{equation}
Then from \cite[Lemma 3.1]{strainSOFT} we clearly have that $\nu (p)\approx C_{\sigma}$ for  $C_{\sigma }>0$.   Furthermore

\begin{lemma}
\label{CCestimate} Let $|\beta |>0$, then $\left\vert \partial _{\beta }\nu
(p)\right\vert \leq C\left\langle {p}\right\rangle ^{-1}$.
\end{lemma}

\begin{proof}
As in \eqref{nuDEF}, we apply \eqref{gammaA} and \eqref{estAc} with $f_{1}\equiv 1$ in the loss
terms.  
\end{proof}

\begin{proposition}
\label{Aestimate} Let $|\beta |>0$ and fix $\ell \geq 0$. For any small $\eta >0,$ there exists a large $R=R(\eta )>0$ and $C=C(\eta )>0$ such that 
\begin{equation*}
\langle w^{2}_{\ell}\partial _{\beta }\{\nu (p)h\},\partial _{\beta }h\rangle
\geq 
|\partial _{\beta }h|_{\nu, \ell}^{2}
-
\eta \sum_{|\alpha |\leq |\beta |}|\partial _{\alpha }h|_{\nu, \ell}^{2}
-
C_{\eta }|\mathbf{1}_{\leq R}h|_{2}^{2}.
\end{equation*}
Here $\mathbf{1}_{\leq R}(p)$ is the indicator function of the ball of radius $R$ centered at the origin.
\end{proposition}

\begin{proof}
We will prove the desired coercivity estimate for a real valued function $h$
to simplify notation; the result follows trivially for a vector valued
function $h=[h_{+},h_{-}]$. We expand out the inner product as 
\begin{gather}
\langle w^{2}_{\ell}\partial _{\beta }\{\nu (p)h\},\partial _{\beta }h\rangle
=
\int_{{\mathbb{R}}^{\dim}}w^{2}_{\ell }~\nu 
~|\partial _{\beta}h|^{2}
+
\sum_{0<\beta _{1}\leq \beta }C_{\beta _{1}}^{\beta }\int_{{\mathbb{R}}^{\dim }}
w^{2}_{\ell}~\partial _{\beta _{1}}\nu ~\partial _{\beta -\beta
_{1}}h~\partial _{\beta }h  \notag 
\\
=
|\partial _{\beta }h|_{\nu ,\ell }^{2}
+
\sum_{0<\beta _{1}\leq \beta}C_{\beta _{1}}^{\beta }
\int_{{\mathbb{R}}^{\dim }}~dp~w^{2}_{\ell}(p)~\partial _{\beta _{1}}\nu (p)~\partial _{\beta -\beta
_{1}}h(p)~\partial _{\beta }h(p).  \notag
\end{gather}
Here $C_{\beta _{1}}^{\beta }$ is the constant which results from the high order differentiation.
Since $|\beta _{1}|>0$ we have from Lemma \ref{CCestimate} that 
$
\left\vert \partial _{\beta _{1}}\nu (p)\right\vert \leq C\left\langle {p}\right\rangle ^{-1}.
$

Then, for fixed $R>0$, we split the second term above as 
\begin{equation*}
\sum_{0<\beta _{1}\leq \beta }C_{\beta _{1}}^{\beta }\int_{{\mathbb{R}}^{\dim }}~dp~w^{2}_{\ell }~\partial _{\beta _{1}}\nu (p)~\partial _{\beta
-\beta _{1}}h~\partial _{\beta }h=\int_{|p|\leq R}+\int_{|p|>R}.
\end{equation*}
On the unbounded part we use Cauchy-Schwartz as follows 
\begin{gather*}
\sum_{0<\beta _{1}\leq \beta }\int_{|p|\geq R}~dp~w^{2}_{\ell }~\left\vert
\partial _{\beta _{1}}\nu (p)~\partial _{\beta -\beta _{1}}h~\partial
_{\beta }h\right\vert \leq \frac{C}{R}|\partial _{\beta }h|_{\nu ,\ell
}\sum_{0<\beta _{1}\leq \beta }|\partial _{\beta -\beta _{1}}h|_{\nu ,\ell }
\\
\leq \frac{C}{R}\sum_{\beta _{1}\leq \beta }
|\partial _{\beta _{1}}h|_{\nu,\ell }^{2}.
\end{gather*}%
On the bounded region we use the compact interpolation of Sobolev-spaces 
\begin{gather*}
\int_{|p|\leq R}dp~
 \sum_{0<\beta _{1}\leq \beta }\left\vert \partial
_{\beta -\beta _{1}}h~\partial _{\beta }h\right\vert 
\leq 
\int_{|p|\leq R}
\left\{
\sum_{0<\beta _{1}\leq \beta }~dp~\left\vert \partial _{\beta -\beta
_{1}}h\right\vert ^{2}+\tilde{\eta} \left\vert \partial _{\beta }h\right\vert ^{2} 
\right\}
\\
\leq 
\eta ^{\prime }\sum_{|\alpha |=|\beta |}\int_{|p|\leq R}~dp~\left\vert
\partial _{\alpha }h\right\vert ^{2}+C_{\eta ^{\prime }}\int_{|p|\leq
R}~dp~\left\vert h\right\vert ^{2} \\
\leq \eta \sum_{|\alpha |\leq |\beta |}|\partial _{\alpha }h|_{\nu ,\ell
}^{2}+C_{\eta }|\mathbf{1}_{\leq R}h|_{2}^{2}.
\end{gather*}
Above the different $\eta>0$ variables are allowed to be arbitrarily small.
This completes the desired estimate. 
\end{proof}

\begin{proposition}
\label{Kestimate} Let $|\beta |\ge 0$ and $\ell \geq 0$. For any small $\eta >0$,  $\exists
C_{\eta }>0$ and $R=R(\eta)>0$ such that the operator from \eqref{kOPdef} satisfies the following estimate:
\begin{equation*}
|\langle w^{2}_{\ell }\partial _{\beta }K(h_{1}),\partial _{\beta }h_{2}\rangle
|\lesssim 
\left\{ \eta \sum_{|\alpha |\leq |\beta |}\left\vert \partial _{\alpha }h_{1}\right\vert _{\nu,\ell}+C_{\eta }\left\vert \mathbf{1}%
_{\leq R}h_{1}\right\vert _{2}\right\} \left\vert \partial _{\beta
}h_{2}\right\vert _{\nu, \ell}.
\end{equation*}
\end{proposition}

\begin{proof}
For $|\be| =0$, this follows from \cite[Lemma 3.3]{strainSOFT}.  We then explain how to prove the case with $|\be| > 0$.
To do this we split $K$ into two parts.  We apply \eqref{gammaA} and \eqref{estAc} with $f_{2}=\sqrt{J}$ and $f_{1}=h_1$ 
as in the definition of $K$ from \eqref{kOPdef}, etc.    

For the term as in \eqref{estAc},
we deduce that it is bounded by a different linear kernel with $v_{\o}J^{\frac{1}{4}}(\QrQ)$ and $v_{\o}J^{\frac{1}{4}}(q).$ Such a new linear operator has the same property as the original $K$ in the center of
mass frame with $v_{\o}J^{\frac{1}{2}}(\QrQ)$ and $v_{\o}J^{%
\frac{1}{4}}(q).$ So the Lemma follows from the result for $K$ without
derivatives in \cite[Lemma 3.3]{strainSOFT}.

For (\ref{gammaA}), the term with $\beta _{1}=\beta $ is bounded by 
\begin{multline*}
\int_{{\mathbb{R}}^{\dim }\times {\mathbb{S}}^{2}}~\frac{s B(p,q,\omega)}{\pZ \qZ }
dqd\omega ~~J^{\frac{1}{4}}(q)~
\left\{
|(\partial _{\beta }h_{1})(\rP)J^{\frac{1}{4}}(\rQ)| 
+
|J^{\frac{1}{4}}(\rP) (\partial _{\beta }h_{1})(\rQ)| 
\right\}
\\
+\int_{{\mathbb{R}}^{\dim }\times {\mathbb{S}}^{2}}~\frac{s B(p,q,\omega)}{\pZ \qZ }
dqd\omega ~~J^{\frac{1}{4}}(q)~|(\partial _{\beta }h_{1})(p)J^{\frac{1}{4}
}(q)|.
\end{multline*}
This new linear operator has the same property as the original $K$. So the lemma follows in this case
from the result for $K$ in \cite[Lemma 3.3]{strainSOFT}.  On the other
hand, for terms with $\beta _{1}<\beta ,$ by compact Sobolev imbedding, it
suffices to consider the case where $|p|$ is large, for which the fast decay
factor $e^{-\frac{1}{8}|p|^{1/m}}$ in (\ref{gammaA}) provides the small
constant $\eta $ and we complete the proof.
\end{proof}

\section{Global solution and rapid decay}

\label{sec2}

In this final section, we explain how to use the new estimates from the
previous sections to prove the global existence and rapid decay of nearby
Maxwellian classical solutions to the relativistic Vlasov-Maxwell-Boltzmann system \eqref{rvmlC} and
\eqref{maxwellC} with \eqref{constraintC}. We prove global
existence following the approach from \cite{MR2000470,MR2100057}. The decay
follows as in the method described in \cite{MR2209761}. Since several of
these previously elucidated details \cite{MR2000470,MR2100057,MR2209761} are
similar, we will simply write down the main steps and refer to the prior
results for an elaboration of the full argument.  We aim to make our argument completely precise in the sense that we refer to the exact argument which is needed from the previous work \cite{MR2000470,MR2100057,MR2209761}. 

First in Section \ref{sec:local} we explain the local existence argument.
Then in Section \ref{positivityofL} we exposit the argument for proving the
crucial positivity of the linearized collision operator for solutions to the relativistic Vlasov-Maxwell-Boltzmann system \eqref{rvmlC} and \eqref{maxwellC} with \eqref{constraintC}. Finally in Section \ref{globalsection} we explain how these estimates can be used to prove the
global in time existence and rapid decay.

\subsection{Local solutions}\label{sec:local}

We now sketch the procedure for obtaining a unique local-in time solution to
the relativistic Vlasov-Maxwell-Boltzmann system \eqref{rvmlC}, %
\eqref{maxwellC}, \eqref{constraintC}. These arguments are rather
standard \cite{MR2000470,MR2100057}.

Given a solution $[f(t,x,p),E(t,x),B(t,x)]$ to the relativistic
Vlasov-Maxwell-Boltzmann system we recall the definition of the instant energy functional 
$\CE_{N,\ell}(t)$ in \eqref{def.eNmN}.  
We furthermore define the dissipation rate $\CD_{N,\ell }(t)$ as ($\ell \geq 0$) 
\begin{equation}
\CD_{N,\ell }(t)\eqdef \sum_{|\gamma |+|\beta
|\leq N}\| \partial _{\beta }^{\gamma }f(t)\|_{\nu,\ell}^{2}.
\notag
\end{equation}%
We state the following local existence theorem.

\begin{theorem}
\label{local} Fix $\ell \ge 0$. $\exists M_0>0$, $T^{*}>0$ such that if $%
T^{*}\le M_0/2$ and 
\begin{equation*}
\CE_{N,\ell}(0)\le M_0/2,
\end{equation*}
then there is a unique solution $[f(t,x,p),E(t,x),B(t,x)]$ to the
relativistic Vlasov-Maxwell-Boltzmann system \eqref{rvmlC}, \eqref{maxwellC} and \eqref{constraintC} on 
$[0,T^{*})\times {\mathbb{T}}_x^3\times {\mathbb{R}}_p^3$ such that 
\begin{equation*}
\sup_{0\le t\le T^{*}} \left\{\CE_{N,\ell}(t) + \int_0^t ds ~\CD_{N,\ell}(s) \right\} \le M_0.
\end{equation*}
The high order energy norm $\CE_{N,\ell}(t)$ is continuous over $%
[0,T^{*}).$ 

If 
$
F_0(x,p)=J +J^{1/2}f_0\ge 0,
$
then $F(t,x,p)=J+J^{1/2}f(t,x,p)\ge 0$. Furthermore, the conservation laws (%
\ref{ma}), (\ref{mo}), and (\ref{en}) hold for all $0<t<T^{*}$ if they are valid
initially (at $t=0$).
\end{theorem}

This local existence theorem can be proven in the standard way using the
estimates in this paper combined with the local existence proof in \cite%
{MR2100057} for the relativistic Landau-Maxwell system. The positivity of
solutions follows from the proof in \cite{MR2000470}.

\subsection{Positivity of L}

\label{positivityofL} In this subsection we elucidate the positivity of the
linearized operator \eqref{L}, $L$, for any small amplitude solution $%
[f(t,x,p),E(t,x),B(t,x)]$ to the full relativistic Vlasov-Maxwell-Boltzmann
system \eqref{rvmlC}, \eqref{maxwellC} and \eqref{constraintC}.

Our main result in this section is as follows.

\begin{theorem}
\label{positive}Let $[f(t,x,p),E(t,x),B(t,x)]$ be a classical solution to (%
\ref{rvmlC}) and (\ref{maxwellC}) satisfying (\ref{constraintC}), (\ref{ma}%
), (\ref{mo}) and (\ref{en}). Then there exists an $M_{0}>0$ and a $\delta
_{0}=\delta _{0}(M_{0})>0$ such that if $N\geq 4$ and 
\begin{equation}
\sum_{|\gamma |\leq N}\left\{ \frac{1}{2}||\partial ^{\gamma
}f(t)||^{2}+||\partial ^{\gamma }E(t)||^{2}+||\partial ^{\gamma
}B(t)||^{2}\right\} \leq M_{0},  \label{m0}
\end{equation}%
then 
\begin{equation*}
\sum_{|\gamma |\leq N}\left( L\partial ^{\gamma }f(t),\partial ^{\gamma
}f(t)\right) \geq \delta _{0}\sum_{|\gamma |\leq N}||\partial ^{\gamma
}f(t)||^{2}.
\end{equation*}
\end{theorem}

In the rest of this section, we always work exclusively with a classical
solution $[f(t,x,p),E(t,x),B(t,x)]$ to (\ref{rvmlC}) and (\ref{maxwellC}).
This argument proceeds, as is customary, via a careful study of the the six
dimensional null space of $L$, for any fixed $(t,x)$; this null space is given by ($1\le i\le 3$)   
\begin{equation}
\nullSpace \eqdef  \mathrm{span}\left\{[\sqrt{J}%
,0], [0,\sqrt{J}], [p_i\sqrt{J},p_i \sqrt{J}],\;[\pZ \sqrt{J}, \pZ  \sqrt{J%
}]\right\}.  \label{null}
\end{equation}
Define the orthogonal projection from $L^2(\mathbb{R}^3_p)$ into $\nullSpace$ by $\mathbf{P}$. Then decompose 
\begin{equation*}
f=\mathbf{P}f+\{\mathbf{I-P}\}f.
\end{equation*}
It is now standard to call $\mathbf{P}f=[\mathbf{P}_+f, \mathbf{P}_-f]\in 
\mathbb{R}^2$ the ``hydrodynamic part'' of $f$ and $\{\mathbf{I-P}\}f=[\{%
\mathbf{I-P}\}_+f, \{\mathbf{I- P}\}_-f]$ the ``microscopic part.'' By
separating its linear and nonlinear part, and using $L_\pm\left(\mathbf{P}
f\right)=0$, we can express the hydrodynamic part of $f$ through the
microscopic part up to a higher order term $h(f)$: 
\begin{equation}
\left\{\partial _t+\frac{p}{\pZ }\cdot \nabla _x\right\}\mathbf{P}_{\pm}f
\mp \left\{E\cdot \frac{p}{\pZ }\right\}\sqrt{J} = l_\pm(\{\mathbf{I-P}\}
f)+h_\pm(f).  \label{macro}
\end{equation}
This is a decomposition of \eqref{rvmlC} (with normalized constants), where 
\begin{eqnarray}
l_\pm(\{\mathbf{I-P}\} f) \eqdef  -\left\{\partial
_t+\frac{p}{\pZ }\cdot \nabla _x\right\}\{\mathbf{I-P}\}_\pm f -
L_\pm\left(\{\mathbf{I-P}\} f\right),  
\label{lconst} \\
h_\pm(f) \eqdef  \mp \left(E+\frac{p}{\pZ }\times
B\right)\cdot \nabla _p f_\pm \pm \left\{E\cdot \frac{p}{\pZ }%
\right\}f_\pm+\Gamma_\pm (f,f).  \label{hconst}
\end{eqnarray}
We further expand $\mathbf{P}_\pm f$ as a linear combination of the basis in
(\ref{null}) as 
\begin{equation}
\mathbf{P}_\pm f =
\left\{a_{\pm}(t,x)+\sum_{j=1}^3b_j(t,x)p_j+c(t,x)\pZ  \right\}\sqrt{J(p)}.
\label{p0}
\end{equation}
The positivity of $L$ is obtained via a careful study of the relativistic
system of macroscopic equations (\ref{bi}) - (\ref{adot}); this system was
derived in \cite{MR2100057}.

We will sketch the derivation of (\ref{bi}) - (\ref{adot}) for the convenience of the reader. Expand the left side of (%
\ref{macro}) with respect to the terms in (\ref{p0}) as 
\begin{equation*}
\left\{\partial^0 a_\pm+ \frac{p_j}{\pZ }\left\{ \partial^j a_{\pm}\mp
E_j\right\} +\frac{p_jp_i}{\pZ }\partial^i b_j+p_j\left\{\partial^0
b_j+\partial^j c\right\} +\pZ  \partial^0 c\right\}\sqrt{J(p)}.
\end{equation*}
Here $\partial ^0=\partial _t$ and $\partial ^j=\partial _{x_j}$. For fixed (%
$t,x$), this is an expansion of left side of (\ref{macro}) with respect to
the basis of $\{e_k\}$, whose components are given by ($1\le i,j\le 3)$ 
\begin{gather}
[\sqrt{J},0], [0,\sqrt{J}], [p_j\sqrt{J}/\pZ ,0], [0,p_j\sqrt{J}/\pZ ], p_j%
\sqrt{J}[1,1], p_jp_i\sqrt{J}/\pZ  [1,1], \pZ \sqrt{J} [1,1].  \notag
\end{gather}
We expand the right side of (\ref{macro}) with respect to the same basis and
compare the coefficients on both sides to obtain the macroscopic equations: 
\begin{eqnarray}
&&\partial ^0c=l_c+h_c,  \label{bi} \\
&&\partial ^i c+\partial^0 b_i=l_i+h_i^{},  \label{c} \\
&&(1-\delta_{ij})\partial ^ib_j+\partial ^jb_i=l_{ij}+h_{ij},  \label{bij} \\
&&\partial^ia_{\pm }\mp E_i=l_{ai\pm }+h_{ai\pm},  \label{ai} \\
&&\partial ^0a_{\pm }=l_{a\pm }+h_{a\pm }.  \label{adot}
\end{eqnarray}
To ease the notation we define the following index set 
\begin{equation*}
\mathcal{M} \eqdef  \left\{c, ~i,~ ij, ~ai\pm,~ a\pm
\left| ~ i, j = 1,2,\dim\right. \right\}.
\end{equation*}
Thus $\mathcal{M}$ is the collection of all indices in the macroscopic
equations. For $\lambda \in \mathcal{M}$ each $l_\macroCOE (t,x)$ are the
coefficients of $l(\{\mathbf{I-P}\}f)$ with respect to the basis elements $%
\{e_k\}$; similarly for each $h_\macroCOE(t,x)$. Precisely, a given $%
l_\macroCOE$ can be expressed as 
\begin{equation*}
l_\macroCOE = \sum_{k} C_k^\macroCOE \langle [l_+(\{\mathbf{I-P}\}f),l_-(\{%
\mathbf{I-P}\}f)], e_k \rangle, \quad C_k^\macroCOE \in {\mathbb{R}}.
\end{equation*}
Also the $h_\macroCOE(t,x)$ can be computed similarly.

From (\ref{p0}) and (\ref{maxwellC}) 
(with normalized constants) we see that 
\begin{eqnarray}
\mathcal{J}=-\int_{{\mathbb{R}}^\dim} [p\sqrt{J }/\pZ ,-p\sqrt{J }/\pZ ]
\cdot \{\mathbf{I-P}\}fdp.  \label{pf}
\end{eqnarray}
Next, we explain how to estimate each of these terms.

\begin{proposition}
\label{linear} Suppose that $N\geq 4$, then we have 
\begin{equation*}
\sum_{|\gamma |\leq N-1}\left( \sum_{\lambda \in \mathcal{M}}\Vert \partial
^{\gamma }l_{\lambda}\Vert +\Vert \partial ^{\gamma }\mathcal{J}\Vert
\right) \lesssim \sum_{|\gamma |\leq N}\Vert \{\mathbf{I-P}\}\partial
^{\gamma }f\Vert .
\end{equation*}
\end{proposition}

This Proposition \ref{linear} is proven for the relativistic Landau-Maxwell
system in \cite[Lemma 11]{MR2100057}. Because the structure of the operator $%
l_\macroCOE$ is similar, the proof of Proposition \ref{linear}
is exactly the same as \cite[Lemma 11]{MR2100057}. The only difference
between these cases is for the term $L_\pm$ in $l_\macroCOE$, since the
operator $L_\pm$ is in fact very different. However because of \eqref{L} the required estimate for $\left< {L\{\mathbf{I-P}%
\}\partial^\gamma f,e_k} \right> $ used in \cite[Lemma 11]{MR2100057} is
supplied by our Theorem \ref{thm:nonLINest}.  We then refer to \cite[Lemma 11]{MR2100057} for the rest.

We now estimate coefficients of the higher order term $h(f)$ from %
\eqref{hconst}.

\begin{proposition}
\label{high}Let (\ref{m0}) be valid for some $M_{0}>0$. Then 
\begin{equation*}
\sum_{|\gamma |\leq N}\sum_{\lambda \in \mathcal{M}}\Vert \partial ^{\gamma
}h_{\lambda}\Vert \leq C\sqrt{M_{0}}\sum_{|\gamma |\leq N}\Vert \partial
^{\gamma }f\Vert .
\end{equation*}
\end{proposition}

As in the previous proposition, Proposition \ref{high} is proven for the
relativistic Landau-Maxwell system in \cite[Lemma 12]{MR2100057}. Again the
structure of $h_\macroCOE$ is similar in both cases, meaning that the proof
of Proposition \ref{high} the same as \cite[Lemma 12]{MR2100057}. The only
difference between these cases is for the term $\Gamma_\pm$ in $h_\macroCOE$%
; the operator $\Gamma_\pm$ is in again quite different. However the needed
estimate for $\left< {\partial^\gamma\Gamma(f,f),e_k} \right> $ used in \cite%
[Lemma 12]{MR2100057} is supplied by our Theorem \ref{thm:nonLINest}.
Otherwise the proof is exactly the same, and for the full details we refer
to the proof in \cite[Lemma 12]{MR2100057}.

Next we estimate the electromagnetic field $[E(t,x),B(t,x)]$ in terms of $%
f(t,x,p)$ through the macroscopic equation (\ref{ai}) and the Maxwell system
\eqref{pf}.

\begin{proposition}
\label{field}Let $[f(t,x,p),E(t,x),B(t,x)]$ be the solution to (\ref{rvmlC}%
), (\ref{maxwellC}) and (\ref{constraintC}) constructed in Theorem \ref%
{local} with the constant \eqref{bar}. Let the small amplitude assumption (\ref{m0}) be valid for some $%
0<M_{0}\leq 1$. Then we have
\begin{equation*}
\sum_{|\gamma |\leq N-1}
\left( 
||\partial ^{\gamma }E(t)||+||\partial ^{\gamma
}\{B(t)-\bar{B}\}||
\right)
\lesssim \sum_{|\gamma |\leq N}\Vert \partial ^{\gamma
}f(t)\Vert .
\end{equation*}
\end{proposition}

Similar to the previous propositions, this Proposition \ref{field} is proven
in exactly the same way as \cite[Lemma 13]{MR2100057} except that we replace
the estimates in the proof of \cite[Lemma 13]{MR2100057} with their
corresponding analogues herein. Specifically, we follow directly the proof
of \cite[Lemma 13]{MR2100057} however we replace the use of \cite[Lemma 11
and Lemma 12]{MR2100057} with Propositions \ref{linear} and \ref{high}
respectively.

\bigskip

Collecting the previous estimates in this section, we can now prove the
crucial positivity of $L$ from \eqref{L}, as stated in Theorem \ref{positive}, for a small amplitude solution $[f(t,x,p),E(t,x),B(t,x)]$. Once again the proof is the same
as the analogous proof for the relativistic Landau-Maxwell system from \cite[Theorem 2]{MR2100057}. We need only replace the estimates used in \cite[Theorem 2]{MR2100057} with their analogues in this section; in particular we
replace \cite[Lemma 13]{MR2100057} with our Proposition \ref{field}, \cite[Lemma 12]{MR2100057} with Proposition \ref{high}, and \cite[Lemma 11]{MR2100057} with Proposition \ref{linear}. We then refer to \cite{MR2100057}
since otherwise the details are exactly the same. This completes Theorem \ref{positive}.

\subsection{Global Solutions}

\label{globalsection} In this section we establish Theorem \ref{mainTHM}. We
first explain how to derive a refined energy estimate. We use the instant
energy functional 
\begin{equation*}
\CE_{m,\ell}(t)\approx 
\sum_{|\beta |\leq m}\sum_{|\gamma |\leq N-|\beta |}
\| \partial _{\beta}^{\gamma }f(t)\|_{2,\ell}^{2}
+
\sum_{|\gamma |\leq N}\| \partial ^{\gamma }[E(t),B(t)]\|^{2}.
\end{equation*}
We also define the refined dissipation rate as 
\begin{equation*}
\CD_{m,\ell }(t)
\eqdef 
\sum_{|\beta |\leq m}\sum_{|\gamma |\leq N-|\beta |}
\| \partial _{\beta}^{\gamma}f(t)\|_{\nu,\ell}^{2}.
\end{equation*}
Here $0\leq m\leq N$. In these spaces we have the estimate:

\begin{proposition}
\label{energy} Fix $\ell \geq 0$. Let $[f(t,x,p),E(t,x),B(t,x)]$ be the
unique solution constructed in Theorem \ref{local} which also satisfies the
conservation laws (\ref{ma}), (\ref{mo}) and (\ref{en}). Let the small
amplitude assumption (\ref{m0}) be valid. For any given $0\leq m\leq N$ and $%
|\beta |\leq m,$ there are constants $C_{m,\ell }^{\ast }>0$ and $\delta
_{m,\ell }>0\,$ such that 
\begin{equation*}
\frac{d}{dt}\CE_{m,\ell}(t)+\delta _{m,\ell}\CD_{m,\ell }(t)\leq C_{m,\ell }^{\ast }\sqrt{\CE_{N,\ell}(t)}~
\CD_{N,\ell}(t).
\end{equation*}
\end{proposition}

We point out that the proof of Proposition \ref{energy} is exactly the same
as the corresponding proof in \cite[Lemma 14]{MR2100057} for the
relativistic Landau-Maxwell system. The differences are that we use the
relativistic Boltzmann estimates from this paper, instead of the
corresponding estimates from \cite{MR2100057}, and secondly that we include
the weight $\ell \ge 0$. For the estimates, we specifically replace \cite[%
Theorem 4]{MR2100057} with our Theorem \ref{thm:nonLINest}, \cite[Lemma 13%
]{MR2100057} with Proposition \ref{field}, \cite[Lemma 7]{MR2100057} with
Propositions \ref{Aestimate} and \ref{Kestimate} and otherwise the argument
follows exactly the proof of \cite[Lemma 14]{MR2100057}. To include the
weights $\ell \ge 0$, we refer to the argument used to prove \cite[Eq (4.6)]%
{DS-VMB}.

\bigskip

Finally we prove the global existence of solutions to the relativistic
Vlasov-Maxwell-Boltzmann system (\ref{rvmlC}) and (\ref{maxwellC}). Notice
that using the estimates above, in particular Proposition \ref{energy}, this
follows from the standard continuity argument as in for example \cite{MR2000470,MR2100057}. Thus we have proven all of Theorem \ref{mainTHM},
except for the decay rates.  But the decay rates in this case follow directly
using the interpolation procedure from \cite{MR2209761} (which was applied
to the Newtonian Vlasov-Maxwell-Boltzmann system and the relativistic
Landau-Maxwell system in \cite{MR2209761}).
Note precisely that for the rapid decay, we use the proof from \cite[Section 2]{MR2209761} combined with the differential inequality from Proposition \ref{energy} in this paper.
\hfill {\bf Q.E.D.}

\subsection*{Acknowledgments}  We would like to thank the referees for their careful reading of the paper and constructive comments which helped to improve the presentation.

\end{document}